\documentclass[a4paper]{article}
\usepackage{amsmath}
\usepackage{amsfonts}
\usepackage{amssymb}
\usepackage{graphicx}
\usepackage{amsthm}
\usepackage{ascmac}
\usepackage{comment}
\usepackage{ulem}
\usepackage{color}

\theoremstyle{definition}
\newtheorem{theorem}{Theorem}[section]
\newtheorem{corollary}{Corollary}[section]
\newtheorem{defn}{Definition}[section]
\newtheorem{remark}{Remark}[section]
\newtheorem{lemma}{Lemma}[section]
\newtheorem{proposition}{Proposition}[section]
\newtheorem{assumption}{Assumption}[section]
\newtheorem{ex}{Example}[section]

\newcommand{\ncd}{\newcommand}

\newcommand{\beqn}{\begin{equation}}
\newcommand{\eeqn}{\end{equation}}

\ncd{\bpf}{\begin{proof}}
\ncd{\epf}{\end{proof}}

\ncd{\ds}{\displaystyle}

\ncd{\z}{{\mathbb Z}}
\ncd{\zd}{{\mathbb Z}^d}
\ncd{\n}{{\mathbb N}}
\ncd{\nd}{{\mathbb N}^d}
\ncd{\R}{{\mathbb R}}
\ncd{\rd}{{\mathbb R}^d}
\ncd{\Q}{{\mathbb Q}}

\ncd{\calA}{{\mathcal A}}
\ncd{\calB}{{\mathcal B}}
\ncd{\calC}{{\mathcal C}}
\ncd{\calD}{{\mathcal D}}
\ncd{\calE}{{\mathcal E}}
\ncd{\calF}{{\mathcal F}}
\ncd{\calG}{{\mathcal G}}
\ncd{\calH}{{\mathcal H}}

\ncd{\calK}{{\mathcal K}}
\ncd{\calL}{{\mathcal L}}
\ncd{\calM}{{\mathcal M}}

\ncd{\calP}{{\mathcal P}}

\ncd{\calR}{{\mathcal R}}
\ncd{\calS}{{\mathcal S}}
\ncd{\calT}{{\mathcal T}}

\ncd{\calZ}{{\mathcal Z}}

\ncd{\bfA}{{\mathbf A}}
\ncd{\bfB}{{\mathbf B}}
\ncd{\bfM}{{\mathbf M}}
\ncd{\bfX}{{\mathbf X}}
\ncd{\bfP}{{\mathbf P}}
\ncd{\bfS}{{\mathbf S}}

\ncd{\prob}{{\mathbb P}}
\ncd{\Ex}{{\mathbb E}}

\ncd{\wh}{\widehat}
\ncd{\wt}{\widetilde}
\ncd{\ol}{\overline}

\newcommand{\indi}{1\hspace{-.2em}{\rm l}}
\ncd{\supp}{{\rm supp}}

\allowdisplaybreaks[4]
\makeatletter

\@addtoreset{equation}{section}
\makeatother

\title{Spread rate of catalytic branching symmetric stable processes}

\author{Yasuhito Nishimori\thanks{
University of Yamanashi, Faculty of Education, 
4-4-37, Takeda, Kofu, Yamanashi, 400-8510, Japan, 
e-mail:
\texttt{y.nishimori@yamanashi.ac.jp}}
\thanks{This work was supported by JSPS KAKENHI Grant Number JP22K03427.}
}

\begin{document}

\maketitle

\noindent
{\bf keywords: Branching $\alpha$-stable process, Spread rate, Kato class measure}  

\begin{abstract}
We study the growth order of the maximal displacement of 
branching symmetric $\alpha$-stable processes. 
We assume the branching rate measure $\mu$ is in the Kato class and 
$\mu$ has a compact support on $\rd$. 
We show that the maximal displacement exponentially grows and 
its order is determined by the index $\alpha$ and the spectral bottom 
of the corresponding Schr\"odinger-type operator. 
\end{abstract}

\section{Introduction}
\subsection{Model and subject}
Let $0 <\alpha < 2$. 
We consider branching symmetric $\alpha$-stable processes with 
splitting on a compact set in $\rd$. 
The branching processes describe stochastic models of 
particle systems with time evolution. 
A particle starts at $x_0 \in \rd $ and 
moves according to the law of a symmetric $\alpha$-stable process 
$\{ X_t , t \ge 0 \}$. 
At an exponential random time $T$, 
the first particle splits into $n$-particles with probability 
$p_n (X_{T-})$. 
Here, the random time $T$ is called 
the first splitting time and $T$ is exponentially distributed 
on the initial particle path such that  
\[
	\prob _{x_0} \left( T > t \mid X_s , s \ge 0 \right)
	= 
	e^{-A_t ^\mu} , 
\]
where $A_t ^\mu$ is the positive continuous additive functional 
(PCAF for short and see \cite[Section 5.1]{FOT11} for detail) 
associated with the branching rate measure $\mu$.  
Then, 
$\{ \{ p_n (x) \}_ {n \ge 1} \mid x \in \rd\}$ is called 
the offspring distribution and  $p_n (x)$ gives the probability of 
splitting into $n$-particles at $x$, 
where $x$ is the place that the particle splits. 
For example, 
$A_t ^\mu = \int _0 ^t \indi _B (X_s) \, d s$, 
when $\mu $ is given by $ \indi _B (x) \, d x$ 
for the indicator function $\indi _B$ of set $B$ and 
the Lebesgue measure $d x$  on $\rd$. 
Particularly, if $\mu$ is the Lebesgue measure, 
then $A_t ^\mu = t$, i.e.,  
$T$ has the exponential distribution with mean one. 
Since measure $\mu$ controls the frequency of branches, 
$\mu$ is called the `catalyst' in the context of particle systems. 
On the other hand, 
the branching process is spatially homogeneous, 
when the splitting mechanism is independent of the space, 
that is, $\mu = dx$ and $p_n (x) = p_n$ for all $n \ge 1$.  

Let ${\mathcal H}$ denote the Schr\"odinger-type operator: 
\begin{equation}
	\calH:= 
	\frac{1}{2} (- \Delta)^{\alpha/2} - (Q-1) \mu 
	\quad \text{on } \ L^2 (\rd ).
	\label{eq:Sch} 
\end{equation}
Here, $Q (x) = \sum _{n \ge 1} ^\infty n p_n (x)$ and 
$(Q-1) \mu$ is the meaning of $(Q(x)-1) \mu (d x)$. 
We write $\lambda ( (Q-1) \mu )$, 
or simply $\lambda$, for the spectral bottom of $\calH$. 
In this paper, 
we assume that 
the branching rate measure $\mu$ has a compact support and 
$R (x) \mu (d x)$ belongs to 
the Kato class (Definition \ref{def:Kato} below), 
where $R (x) = \sum _{n \ge 1} n(n-1) p_n (x)$. 
The PCAF $A_t ^\mu$ increases only when the initial particle moves 
on the compact support of $\mu$. 
It follows that the branches occur only on this compact region. 
Furthermore, we assume that $\lambda < 0$.

Let us denote the configuration of particles at time $t$ by 
\[
	\bfX _t = \left( \bfX _t ^u , u \in Z_t \right) .
\]
Here, $Z_t$ is the set of all particles at time $t$ and  
$\bfX_t ^u \in \rd$ is the meaning of the position of particle $u \in Z_t$. 
We define the maximal displacement at time $t$ by 
\begin{equation*}
	L_t := 
	\sup _{u \in Z_t} \left| \bfX_ t ^u \right| . 
	\label{Lt}
\end{equation*}
When $d=1$, 
this corresponds to the largest distance from the origin of either the rightmost 
or leftmost particles, respectively.
Our interest is the pathwise time evolution of $L_t$ as $t \to \infty$.

\subsection{Background and motivation}
We can consider several branching Markov processes through \eqref{eq:Sch}. 
The first term of $\calH$ represents the law of the one-particle motion and 
the second term controls the branching rule (cf. \cite{INW-I, INW-II, INW-III}). 
The branching Brownian motion (BBM for short) 
is $\alpha = 2$ in \eqref{eq:Sch}. 
In particular, 
the BBMs are the basic models, 
if $d=1$, $Q \equiv 2$ and the branching rate measure is 
given by the Lebesgue measure or the Dirac measure $\delta _0$. 
Since we can consider that $\mu$ contains $Q-1$ in \eqref{eq:Sch}, 
we here ignore $Q-1$ and we treat only ${\mathcal H} = \frac{1}{2}(- \Delta)^{\alpha /2} - \mu$
to describe our background, conveniently.

McKean \cite{Mc75} showed that 
the distribution function of the rightmost particle 
is a solution of the F-KPP equation and 
the distribution function, 
which is scaled by the space factor, 
converges to a solution of a traveling wave equation, 
for the spatially homogeneous BBM. 
This research is the point of departure for 
the rightmost (and leftmost) particle and 
the maximal displacement $L_t$. 
The scaling factor was clarified by Bramson \cite{B78}. 
His result also precisely indicates the pathwise time evolution of the spread rate. 
Then, on inhomogeneous BBMs, 
the relationships between the spread rate and $\mu$ 
were studied by several researchers. 
There are mainly two ways of investigating the growth rate of the maximal displacement
(or the rightmost particle). 
One is to investigate the pathwise growth order of $L_t$, 
the other is to determine the scaling factor $R(t)$ such that 
the distribution of $L_t - R(t)$ converges to a certain distribution. 
If $L_t - R (t)$ converges to the distribution, 
then we can regard that $R(t)$ approximates $L_t$ in distribution. 
Erickson \cite{E84} investigated the case of $\mu = V(x) \,d x$, 
where $V$ degenerates to zero at infinity. 
He revealed the pathwise time evolution of the rightmost particle. 
He showed that the growth is a time linear 
and its coefficient is determined by the eigenvalue 
of the corresponding Schr\"odinger operator. 
Then, Bocharov and Harris \cite{BH14} researched about  
$\mu = \beta \delta _0$, $\beta >0$. 
Shiozawa \cite{S18} extended their results to the maximal displacement $L_t$ 
for the BBMs on $\rd$ ($d \ge 1$), 
when $\mu$ belongs to the Kato class and 
it has a compact support. 
Bocharov and Wang \cite{BW} studied one-dimensional 
spatially homogeneous BBMs adding a point catalyst 
($\mu = \beta dx + \beta _0 \delta _0$ for positive constants $\beta $ and $\beta _0$). 
These \cite{BH14, S18, BW} 
are concerned with the pathwise evolution of $L_t$, 
and these are classified into \cite{E84}. 
On the other hand, 
Lalley and Sellke \cite{LS88} researched about 
the limiting distribution of rightmost particle 
for the similar setting to \cite{E84}. 
Lalley and Sellke \cite{LS89} also investigated the case of  
$\mu = (b + \beta (x)) \, d x$, 
where $b $ is a positive constant and 
$\beta (x)$ is a non-negative, continuous, integrable function. 
Then, Bocharov and Harris \cite{BH16} for $\mu = \beta \delta _0$, and 
Nishimori and Shiozawa \cite{ShioandN} for the Kato class mesure, 
computed the appropriate $R(t)$ 
and determined the limiting distribution of $L_t - R(t)$, respectively. 
As a matter of course, 
the coefficient of leading term of $R(t)$ coincides with 
a.s. limit of the spread rate $L_t /t$ as $ t \to \infty$. 
We can confirm this consistency 
between \cite{BH14} and \cite{BH16}, 
between \cite{S18} and \cite{ShioandN}, respectively.

Recently, branching processes with the Kato class measures 
as branching rate measures, 
including singular measures with respect to the Lebesgue measures 
such as the Dirac measure, 
have been intensively studied. 
The trigger is \cite{BH14} by Bocharov and Harris. 
They computed the expectation of the first moment of the number of particles 
for the one-dimensional BBM with a point catalyst at the origin 
by using the Many-to-One lemma (cf. \cite{HaH}). 
According to this lemma, 
they altered the moment to 
the Fynman-Kac functional given by the local time at the origin. 
Since it is well known that the marginal distribution of the Brownian motion 
and the local time, by using these, 
Bocharov and Harris directly computed the above moment 
and its asymptotic behavior. 
As a result, they proved the pathwise spread rate of the rightmost particle. 

By \eqref{eq:Sch}, 
the branching $\alpha$-stable process with the Kato class measure, 
which has a compact support, 
is a natural extension of BBMs 
as in \cite{BH14, S18}. 
Recently, Shiozawa \cite{S22} obtained the limiting distribution of 
$L_t - R _\alpha(t)$ 
for this branching symmetric $\alpha$-stable processes. 
He improved analytic tools for the moment calculations 
and revealed the asymptotic behavior of $L_t$ in distribution. 
Moreover, his result shows that 
the asymptotic behavior of $L_t$ is different from the branching Brownian case. 
Motivated by his work, 
we determine the pathwise growth order of $L_t$ as $t \to \infty$, 
for the branching $\alpha$-stable processes. 
Our argument bases on \cite{BH14, S18}. 
Here, the moment calculation and its asymptotic order, 
which were given by Shiozawa \cite{S22}, 
play important roles. 

The symmetric $\alpha$-stable processes have 
the `heavy-tail' in the meaning of 
\eqref{eq:shio3} (see also \eqref{eq:distrXt}). 
On the other hand, 
the Brownian motions have the `light-tail', 
that is, these tail distributions exponentially decay. 
Since each particle of the branching stable process 
spreads more rapidly than the BBM, 
we easily expect that $L_t$ grows faster for the branching stable process 
than the BBM. 
Our standpoint is to appear the asymptotic behavior of $L_t$, 
for the branching stable process as the heavy-tailed process. 
Our results also make clear the difference between 
the BBMs and the branching stable processes 
for the growth order of $L_t$ as $t \to \infty$.  

Our works are focusing on the pathwise growth order of $L_t$, 
and these are classified as \cite{E84}. 
Similar to the contrast between \cite{E84} and \cite{LS88}, 
as for research of $L_t$ in the distribution sense, 
Lalley and Shao \cite{LSh} revealed the thickness of the tail distribution of 
the rightmost particle, 
for the spatially homogeneous branching stable processes. 
Bulinskaya \cite{Bulin2021} (see also \cite{Bulin2018}) 
for the continuous time catalytic branching random walks, 
and Shiozawa \cite{S22} for the aforementioned branching stable processes,  
they determined the appropriate scaling factors and 
the limiting distributions of the maximal displacement, respectively. 
In particular, Ren et al. \cite{RSZ}, 
for the spatially homogeneous branching L\'evy process, 
showed the limiting distribution of all particles ordered from right to left, 
not only the rightmost particle. 
Bhattacharya et al. \cite{Bhatt} showed the convergence of 
the point process associated with a branching random walk 
having regular varying steps 
and they revealed the limiting process.

\subsection{Main results}
We consider the branching $\alpha$-stable process on $\rd$ such that 
the branching rate measure $\mu$ belongs to the Kato class and 
has a compact support. 
In addition, we make Assumption \ref{ass1} on $\mu$. 
Our model corresponds to one, 
which is used by \cite{S22}. 

Our main claim (Corollary \ref{cor}) is that 
the maximal displacement $L_t$ grows exponentially 
and the exponent is determined by $\lambda$ and $\alpha$ as follows: 
\begin{equation}
	\lim _{t \to \infty} 
	\frac1t \log L_t 
	= \dfrac{-\lambda}{\alpha} , \qquad 
	\prob _x (\cdot \mid M_\infty >0)\text{-a.s.}, 
	\label{eq:main-in-intro}
\end{equation}
where $M_\infty$ is the limit of the martingale defined by \eqref{mart} below. 

In the Brownian cases, 
the leading term of the maximal displacement grows in linear time 
(e.g. \cite{BH14, S18}), that is, 
$L_t$ is the same order as $\sqrt{-\lambda /2} t$, 
almost surely, when $t \to \infty$. 
In the stable case, 
since each particle jumps far according to the heavy-tailed distribution, 
$L_t$ spreads faster than the Brownian cases. 
Our result \eqref{eq:main-in-intro} supports this. 
The constant $-\lambda / \alpha$ in \eqref{eq:main-in-intro} 
is consistent with the limit of $t^{-1} \log R_\alpha (t)$, 
where $R_\alpha (t) = (e^{-\lambda t} \kappa)^{1/\alpha}$, 
for $\kappa >0$, 
is the scaling factor, which is mentioned in the previous subsection. 

To show \eqref{eq:main-in-intro}, 
we follow the same argument as in \cite{BH14, S18} 
for the Brownian cases. 
Let $B(R)$ be the sphere with radius $R$ centered at the origin. 
We write $N_t ^R$ for the number of particles 
which stay $B (R)^c$. 
We consider $B (\kappa (t))$ the zone where particles remain. 
Theorem \ref{theo:main} provides that
if $\kappa (t)$ is so large, then there may be no particle on $B (\kappa (t))^c$; 
if $\kappa (t)$ is so less, then $N_t ^{\kappa (t)}$ exponentially goes to infinity. 
\begin{theorem}\label{theo:main}
We assume that $\lambda < 0$. 
For $\delta > 0$, 
let us set $\kappa _\delta (t) = e^{\delta t} a(t)$, 
where $a(t) > 0$ is monotone increasing 
and $t^{-1} \log a(t) \to 0$, $t \to \infty$. 
\begin{enumerate}
\item[(i)] For any $\delta > - \lambda / \alpha$ and $x \in \rd$, 
\[
	\lim _{t \to \infty} N_t ^{\kappa _\delta (t)} 
	= 0, \qquad \prob _x \text{-a.s.}
\]

\item[(ii)] 
When $\prob _x (M_\infty > 0) > 0$, 
for any $\delta \in (0, -\lambda /\alpha)$, 
\[
	\lim _{t \to \infty} 
	\frac1t \log N_t ^{\kappa _\delta (t)} 
	= -\lambda - \alpha \delta, \qquad 
	\prob _x (\cdot \mid M_\infty > 0) \text{-a.s.}
\]
\end{enumerate}
\end{theorem} 

By (i), we can suppose that $L_t \le \kappa _{-\lambda / \alpha} (t)$, 
for large $t$.  
By (ii), we can also suppose that $L_t \ge \kappa _{\delta} (t)$, 
for any $\delta > -\lambda / \alpha$. 
By these, 
we will show \eqref{eq:main-in-intro} in Corollary \ref{cor}. 

\bigskip 

The remainder of this paper is constructed as follows: 
In Section 2, we summarize the basic notions and properties 
of the symmetric $\alpha$-stable processes and 
the branching symmetric $\alpha$-stable processes. 
In Section 3, we show Lemmas \ref{lem:3-7}--\ref{lem:liminf} 
for Theorem \ref{theo:main}. 
Similar to \cite{BH14, S18}, 
we use the Borel-Cantelli lemma. 
However, we can not use the spatial homogeneity of the distribution 
of the running maximum for the stable processes. 
Overcoming this, we use the estimate of the tail distribution of 
the running maximum. 
Then, we shall show Theorem \ref{theo:main} and Corollary \ref{cor} 
by the same argument as in \cite{BH14, S18}. 

Throughout this paper, 
the letters $c$ and $C$ (with subscript and superscript) 
denote finite positive constants which may vary from place to place. 
For positive functions $f(t)$ and $g(t)$ on $(0,\infty)$, 
we write $f(t) \lesssim g(t)$, $t \rightarrow \infty$ 
if positive constants $T$ and $c$ exist such that 
$f(t)\le c g(t)$ for all $t \geq T$. 
We write $f (t) \asymp g(t)$ if and only if 
both $f(t) \lesssim g(t)$ and $g(t) \lesssim f(t)$ hold. 
We also write $f(t)\sim  g(t)$, $t \rightarrow \infty$ 
if $f(t)/g(t)\rightarrow 1$ as $t \rightarrow \infty$. 
We will omit ``$t \to \infty$'' for short when the meaning is clear.

\section{Preliminaries}
Let $0 < \alpha <2$. 
In this section, 
we introduce the symmetric $\alpha$-stable process, 
the Kato class measure and the branching symmetric $\alpha$-stable process. 

\subsection{Symmetric $\alpha$-stable processes}

Let $(\{ X_t \}_{t \ge 0}, \{P_x \}_{x \in \rd}, \{\calF _t \}_{t \ge 0})$ 
be the symmetric $\alpha$-stable process $\rd$, 
that is, the Markov process generated by $-\frac12 (-\Delta)^{\alpha/2}$. 
Here, $\{\calF _t\}$ is the minimal augmented admissible filtration. 
We especially write $P$ when $x$ is the origin. 
The following result, for the transition function $p_t (x,y)$, 
was proved by Wada \cite{Wada} 
(see also \cite[Lemma 1]{S22}). 

\begin{lemma}[]\label{lem:Shio}
There exists a positive continuous function $g$ on [$0,\infty)$ 
such that
$$
	p_t(x, y) 
	=
	\dfrac{1}{t^{d/\alpha}}
	g \left( \dfrac{|x-y|}{t^{1/\alpha}} \right).
$$
Moreover, the function $g$ satisfies the following: 
	\begin{equation}
		\lim_{r \to \infty} r^{d+\alpha} g(r) 
		=
		\dfrac{\alpha 2^{\alpha-2} \sin (\frac{\alpha \pi}{2}) 
			\Gamma (\frac{d+\alpha}{2}) \Gamma (\frac{\alpha}{2})}
			{\pi ^{d/2+1}} .
	\label{eq:shio3}
	\end{equation} 
\end{lemma}

The following is well known as the scaling property 
of the symmetric $\alpha$-stable process. 
We note by Lemma \ref{lem:Shio} that, 
for any $\kappa \ge 0$, $t >0$ and $x \in \rd$,  
\begin{equation}
	\begin{aligned}
	& 
	P_x ( |X _t| \ge \kappa) 
	= 
	\int _{|y| \ge \kappa} 
	t^{-d / \alpha} g \left( \frac{|x-y|}{t^{1/\alpha}} \right) 
	\, d y 
	\\
	= &
	\int _{|t^{1/\alpha}z +x| \ge \kappa }
	g (|z|) 
	\, d z 
	= 
	P \left( \left| t^{1/\alpha} X_1 + x \right| \ge \kappa \right) .
	\end{aligned}
	\label{eq:20230206-2}
\end{equation}
By this and 
$ \left\{ y \in \rd : |y| \ge \kappa \right\}
\subset \left\{ y \in \rd  : |y-x| \ge \kappa - |x| \right\} $, 
\begin{equation*}
	\begin{aligned}
	P_x (|X_t| \ge \kappa )
	&= 
	\int _{|y| \ge \kappa} 
	t^{-d / \alpha} g \left( \frac{|x-y|}{t^{1/\alpha}} \right) 
	\, d y 
	\le 
	\int _{|y-x| \ge \kappa -|x|} 
	t^{-d/\alpha} g \left( \frac{|x-y|}{t^{1/\alpha}} \right)
	\, d y
	\\
	&= 
	\int _{t^{1/\alpha} |z| \ge \kappa -|x|} 
	g \left( |z| \right)
	\, d z	
	=
	\omega _d
	\int _{(\kappa -|x|) t^{-1/\alpha}} ^{\infty}
	g \left( r \right) r^{d-1}
	\, d r ,
	\end{aligned} 
	\label{eq:20230206-1}
\end{equation*}
where $\omega _d = 2 \pi ^{d/2} \Gamma (d/2) ^{-1}$ 
is the surface area of the unit ball in $\R ^d$. 
Similarly, by $
	\left\{  
		z \in \rd : t^{1/\alpha} \left|z \right| - |x| \ge \kappa  
	\right\}
	\subset 
	\left\{ 
		z \in \rd : \left| t^{1/\alpha} z + x \right| \ge \kappa
	\right\}$ 
and \eqref{eq:20230206-2}, 
\begin{equation*}
	P_x (|X_t| \ge \kappa )
	=
	\int _{|t^{1/\alpha}z +x| \ge \kappa }
	g (|z|) 
	\, d z 
	\ge 
	\int _{|t^{1/\alpha}z| - |x| \ge \kappa} 
	g(|z|)
	\, d z  
	= 
	\omega _d
	\int _{(\kappa + |x| ) t^{-1/\alpha}} ^{\infty}  
	g(r) r^{d-1} 
	\, d r . 
	\label{eq:20230206-3}
\end{equation*} 
Thus,  
\begin{equation} 
	\omega _d
	\int _{(\kappa + |x| ) t^{-1/\alpha}} ^{\infty}  
	g(r) r^{d-1} 
	\, d r  
	\le 
	P_x (|X_t| \ge \kappa )
	\le 
	\omega _d
	\int _{(\kappa - |x| ) t^{-1/\alpha}} ^{\infty}  
	g(r) r^{d-1} 
	\, d r  .
	\label{eq:distrXt}
\end{equation}

Let us set 
\[
	\calM _t 
	= 
	\sup _{0 \le s \le t} \left| X_s \right| . 
\]
The tail probability of $\calM_t$ is determined by the one of $|X_t|$. 
By the fundamental argument (e.g., \cite[Section 2.8.A]{KS} for the Brownian case), 
we have the following lemma (see Section \ref{appendix1}). 
More general cases appeared in \cite{KR}. 
\begin{lemma}\label{lem:runningmax}
Let  $\kappa >0$. 
For any $t > 0$ and $x \in \rd$ with $|x|< \kappa$, 
\begin{equation}
	P _x ( |X_t| \ge \kappa ) 
	\le 
	P _x ( \calM _t \ge \kappa ) 
	\le 
	2
	P _x ( |X_t| \ge \kappa ) .
	\label{eq:lem}
\end{equation} 
\end{lemma}
Combining Lemma \ref{lem:Shio} with Lemma \ref{lem:runningmax}, 
we have the tail estimate of the running maximum 
for the stable processes. 
\begin{lemma}\label{lem:tailmax}
Functions $\kappa _i (s) $, $i=1,2$ on $[0,\infty)$ 
satisfy that $\kappa_1 (s) < \kappa _2 (s)$, for all $s \ge 0$, 
$\kappa _i (s) \to \infty$ and 
$\kappa _2 (s) - \kappa _1 (s) \to \infty$ as $s \to \infty$. 
Then, there exist positive constants $c_1 , c_2$ and $T$ such that, 
for any $s \ge T$ and $|x| \le \kappa _1 (s) $,  
\begin{equation}
	c_1 (\kappa _2 (s) + \kappa _1 (s)) ^{-\alpha} 
	\le 
	P_x (\calM_ 1 \ge \kappa_2 (s))
	\le  
	c_2 (\kappa _2 (s) - \kappa _1 (s) ) ^{-\alpha}.
	\label{eq:runngingmax}
\end{equation}
\end{lemma}

\begin{proof}
Let $t=1$ on \eqref{eq:lem}.  
It suffices to give upper and lower estimates for 
$P_x (|X_1| \ge \kappa _2 (s))$. 
We see from \eqref{eq:distrXt} that, 
for any $s \ge 0$ and $|x| \le \kappa _1 (s)$, 
\[
	P_x (|X_1| \ge \kappa _2 (s))
	\le 
	\omega _d
	\int _{\kappa _2 (s) - |x|} ^{\infty}  
	g(r) r^{d-1} 
	\, d r 
	\le 
	\omega _d
	\int _{\kappa _2 (s) - \kappa _1 (s)} ^{\infty}  
	g(r) r^{d-1} 
	\, d r . 
\] 
By \eqref{eq:shio3}, there exist positive constants $C, R$ 
such that $g(r) \le C r^{-d-\alpha}$, for $r \ge R$. 
We can take $T >0$ such that 
$\kappa _2 (s) - \kappa _1 (s) \ge R$, for ant $s > T$. 
Thus, for any $s > T$, 
$$
	P_x (|X_1| \ge \kappa _2 (s))
	\le 
	c
	\int _{\kappa _2 (s) - \kappa _1 (s)} ^{\infty}  
	r^{-\alpha -1} 
	\, d r 
	= 
	c'
	\left( 
		\kappa _2 (s) - \kappa _1 (s)
	\right) ^{-\alpha} .
$$ 
Similarly, 
we have the lower estimate of $P_x (|X_1| \ge \kappa _2 (s)) \ge 
c'' ( \kappa _2 (s) + \kappa _1 (s) ) ^{-\alpha}$. 
\end{proof}

\subsection{Kato class measures}\label{sec:Katoclass}

We assume that the branching rate measure is in Kato class 
and has a compact support. 
For the convenience of the reader, 
we repeat the relevant materials from \cite{S22} without the proofs.  

For $\beta > 0$, 
the $\beta$-resolvent density $G_\beta (x,y)$ of 
the symmetric $\alpha$-stable process is given by 
\[
	G_\beta (x,y) 
	= 
	\int _0 ^{\infty}
		e^{-\beta t} 
		p_t (x,y) d t, 
	\quad x,y \in \rd , \quad t > 0 .
\]
\begin{defn}\label{def:Kato}
\begin{itemize}
\item[{\rm (i)}] 
A positive Radon measure $\nu$ on ${\mathbb R}^d$  
is in the Kato class ($\nu\in {\cal K}$ in notation) if 
\[
	\lim_{\beta \to \infty}
	\sup_{x \in \rd}
	\int_{\rd} G_\beta (x,y) 
	\, \nu(d y)
	=0.
\] 
\item[{\rm (ii)}] 
A measure $\nu \in {\cal K}$ is $1$-Green tight 
($\nu \in \calK _\infty (1)$ in notation) if
\[
	\lim_{R \to \infty} 
	\sup_{x \in \rd} 
	\int_{{|y| \ge R}} G_1(x,y) 
	\, \nu(d y)=0.
\] 
\end{itemize}
\end{defn}
Clearly, if Kato class measure $\nu$ has a compact support, 
then $\nu$ is $1$-Green tight. 

Let $(\calE , \calF)$ the regular Dirichlet form generated 
by the symmetric $\alpha$-stable process:
\begin{eqnarray*}
	\calE (u,v)
	&:= &
	\frac{1}{2} \calA (d,\alpha)
	\iint _{\rd \times \rd \setminus \text{diag}} 
		\dfrac{(u(x) - u(y))(v(x) - v(y))}{|x-y|^{d+\alpha}} 
	\, d x d y ,
	\\
	\calF 
	&:=& 
	\left\{
		u \in L^2 (\rd) \ \left| \ \calE (u,u) < \infty \right. 
	\right\} , 
\end{eqnarray*} 
where `diag' $= \{ (x,x) \mid x \in \rd \}$ and 
\[
	\calA (d,\alpha) 
	= 
	\dfrac{\alpha 2^{\alpha-1} \Gamma (\frac{\alpha+d}{2})}
		{\pi ^{d/2} \Gamma (\frac{2-\alpha}2) } 
\]
(see \cite[Examples 1.4.1 and 1.2.1]{FOT11}). 
By the regularity, $u \in \calF$ admits a quasi continuous version $\wt{u}$. 
We always write $u$, instead of $\wt{u}$. 

For $\nu := \nu^+ - \nu^- \in \calK - \calK$, 
we introduce a symmetric bilinear form $\calE ^\nu$ by 
\[
	\calE ^\nu (u,u)
	= 
	\calE (u,u) - \int _{\rd} u^2 d 
	\, \nu, 
	\quad u \in \calF .
\]
Since $\nu$ charges no set of zero capacity (\cite[Therem 3.3]{ABM91}), 
$\calE ^\nu (u,u)$ is determined uniquely by $u$, 
that is, $\calE ^\nu (u,u)$ is unaffected by 
the choices of the quasi continuous versions. 
According to \cite[Theorem 4.1]{ABM91}, 
$(\calE ^\nu, \calF)$ is a lower semi-bounded symmetric closed form. 
Then, we write $\calH ^\nu$ the self-adjoint operator on $L^2 (\rd)$ 
such that $\calE ^\nu (u,v) = ( \calH ^\nu u ,v)$ and 
we write $p_t ^\nu$ for the $L^2$-semigroup generated by $\calH ^\nu$. 
That is, $\calH^\nu = \frac12 (- \Delta)^{\alpha /2} - \nu $. 
By \cite[Theorem 7.1]{ABM91}, 
$p_t ^\nu$ admits a symmetric integral kernel $p_t ^\nu (x,y)$ 
which is jointly continuous on $(0,\infty) \times \rd \times \rd$. 

Let $A_t ^{\nu}$ be PCAF 
which is in the Revuz correspondence with $\nu \in \calK$
(see \cite[pages 230 and 401]{FOT11}). 
We set $A_t ^\nu = A_t ^{\nu^+} - A_t ^{\nu^-}$, 
for $\nu = \nu^+ - \nu^- \in \calK - \calK$. 
By the Feynman-Kac formula, 
\[
	p_t ^\nu f (x) 
	= 
	E_x \left[ e ^{A_t ^\nu} f(X_t) \right], \quad 
	f \in L^2 (\rd) \cap \calB_b (\rd) .
\]

In \cite{S22}, Shiozawa proved the invariance of the essential 
spectrum of $\frac{1}{2}(-\Delta) ^{\alpha /2}$ 
under the  perturbation with respect to the finite Kato class measure
(see also \cite[Remark 8]{S22}). 
Let $\sigma _{\text{ess}} (\calH ^\nu)$ 
be the essential spectrum of $\calH ^\nu$. 

\begin{proposition}[Proposition 4 in \cite{S22}] \label{prop4}
If $\nu ^+$ and $\nu^-$ are finite Kato class measures, 
then $\sigma _{ \text{ess} } ( \calH ^\nu ) 
= \sigma _{ \text{ess} } ( (-\Delta)^{\alpha/2}/2 ) = [0,\infty)$. 
\end{proposition}

We denote by $\lambda (\nu)$ 
the bottom of the $L^2$-spectrum of $\calH ^\nu$: 
\[
	\lambda (\nu) 
	= 
	\inf \left\{ \calE ^\nu (u,u) 
	\ \left| \  
	u \in C_0 ^\infty (\rd), \ 
	\int _{\rd} u^2 
	\, d x = 1
	\right.  \right\} .
\]
Here, $C_0 ^\infty (\rd)$ is the set of all 
infinitely differentiable functions with compact support on $\rd$. 
Proposition \ref{prop4} implies that 
if $\lambda (\nu) < 0$, 
then $\lambda (\nu)$ is the eigenvalue. 
We see from \cite[Theorem 2.8 and Section 4]{T08} that 
the corresponding eigenfunction has a bounded and 
strictly positive continuous version. 
Let us denote by $h$ the $L^2$-normalized version. 

\begin{remark}
Although we introduce the invariance property for general signed measures, 
we use only positive measures. 
In this paper, the branching rate measure $\mu$ and  
$\nu := (Q-1) \mu$ are always positive. 
\end{remark}

In this paper, 
we always assume that the spectral bottom $\lambda ((Q-1) \mu)$ of 
\eqref{eq:Sch} is strictly negative. 
In \cite[Examples 19 and 20]{S22} (see also \cite[Example 4.7]{S08}), 
Shiozawa gave the examples such that $\lambda ((Q-1) \mu ) < 0$. 
To apply his examples, we assume that $p_2 \equiv 1$, 
thus $Q \equiv2$. 
\begin{ex}[Examples 19 and 20 in \cite{S22}]
For $d=1$, $\alpha \in (1,2)$. 
If the branching rate measure $\mu = c \delta _0$ $(c>0)$, then 
\[
 	\lambda ((Q-1) \mu) 
	=
	- 
	\left\{
		\dfrac{c 2^{1/\alpha}}
		{\alpha \sin (\frac{\pi}{\alpha})}
	\right\} ^{\alpha /(\alpha -1)} .
\]
For $d > \alpha$, $1 < \alpha < 2$. 
If $\mu$ is the surface measure $c \delta _r$ ($c, r>0$) on $\{ y \in \rd : \| y \| = r \}$, then 
\[
	\lambda ((Q-1) \mu ) < 0
	\iff 
	r
	> 
	\left\{ 
		\dfrac{\sqrt{\pi} \Gamma (\frac{d+\alpha-2}{2})  \Gamma (\frac{\alpha}{2})}
			{c \Gamma (\frac{d-\alpha}{2}) \Gamma (\frac{\alpha-1}{2})} 
	\right\} ^{1 / (\alpha -1)} .
\]
\end{ex}

Let $\nu$ be a Kato class measure with a compact support. 
We assume that $\lambda (\nu) < 0$. 
We introduce two asymptotic behaviors of the Feynman-Kac functionals.  
By \cite[Theorem 5.2]{T08}, 
\begin{equation}
	E_x \left[ e^{A^\nu _t} \right] 
	\asymp 
	e^{-\lambda (\nu) t}, \quad 
	t \to \infty ,
	\label{eq:3-11}	
\end{equation}
for any $x \in \rd$ 
(for a more precise evaluation, see \cite[Remark 11]{S22}). 
The following asymptotic behavior is given by [Lemma 12 in \cite{S22}]
(see also \cite[(3.38) and (3.39)]{ShioandN}, for the Brownian cases). 
Let $\kappa (t) : [0, \infty ) \to [0, \infty)$. 
We assume that $\kappa (t) t^{-1/\alpha} \to \infty$, 
as $t \to \infty$. 
Then, for each $x \in \rd$,
\begin{equation}
	E_x \left[ e^{A_t ^{\nu}} \; ; \; |X_t| \ge \kappa (t) \right]
	\asymp
	\kappa (t) ^{- \alpha} e^{-\lambda (\nu) t} , 
	\qquad t \to \infty. 
	\label{eq:3-1}
\end{equation}
In the proof of the main theorem, 
we extensively use \eqref{eq:3-11} and \eqref{eq:3-1}.

\subsection{Branching $\alpha$-stable processes}
In this subsection, 
we recall the branching symmetric $\alpha$-stable process
(see \cite{INW-I, INW-II, INW-III}, \cite{S08} and references therein for details) and 
we introduce some properties.

Let $\mu \in \calK$ be a branching rate measure and 
$\{ \{ p_n (x) \}_ {n \ge 1} \mid x \in \rd\}$ 
be a branching mechanism such that  
\[
	0 \le p_n (x) \le 1 , \ n \ge 1 
	\ \text{ and } \ 
	\sum _{n =1} ^{\infty} p_n (x) = 1, \ x \in \rd .
\]
A random time $T$ has an exponential distribution
\[
	\prob _x \left( T > t \mid \calF _{\infty} \right) 
	= e^{-A_t ^{\mu}}, \quad t > 0. 
\]
An $\alpha$-stable particle starts at $x \in \rd$. 
After an exponential random time $T$, 
the particle splits into $n$-particles with probability $p_n (X_{T-})$. 
New ones are independent $\alpha$-stable particles starting at $X_{T-}$ and 
each one independently splits into multiple particles, 
the same as the first one. 
The $n$-particles are represented by 
a point in the following configuration space $\bfS$. 
Let $(\rd)^{(0)} = \{ \Delta \}$ and $(\rd)^{(1)} = \rd$. 
For $n \ge 2$, we define the equivalent relation $\sim$ on 
$(\rd)^n = \underbrace{\rd \times \cdots \times \rd}_{n}$ as follows: 
for $\mathbf{x} ^n = (x^1,\dots , x^n)$ and 
$\mathbf{y} ^n = (y^1,\dots , y^n) \in (\rd ) ^n$, 
we write $\mathbf{x} \sim \mathbf{y}$ 
if there exists a permutation $\sigma $ on $\{1,2,\dots , n\}$ such that 
$y^i = x^{\sigma (i)}$ for any $i \in \{1,2\dots ,n \}$. 
If we define $(\rd)^{(n)}=(\rd)^n \slash \sim $ for 
$n \ge 2$ and $\bfS = \bigcup _{n=0}^{\infty} (\rd)^{(n)}$, 
then $n$-points in $\rd$ determine a point in $(\rd)^{(n)}$. 
The symmetric $\alpha$-stable process 
$(\{ \bfX _t \}_{t \ge 0}, \{ \prob_{\mathbf{x}} \}_{\mathbf{x} \in \bfS}, 
\{ \calG _t \}_{t \ge 0})$ is an $\bfS$-valued Markov process. 
Abusing notation, 
we regard $x \in \rd$ in the same way as $\mathbf{x} \in (\rd)^{(1)}$ and 
write $\prob _x$ for $x \in \rd$. 
That is, $(\{ \bfX _t \}_{t \ge 0} , \prob _x , \{ \calG _t \}_{t \ge 0})$ is 
the branching symmetric $\alpha$-stable process 
such that a single particle starts from $x \in \rd$. 

Let $Z_t$ be the set of all particles and 
$\bfX_t ^u$ the position of $u \in Z_t$ at time $t$. 
For $f \in \calB _b (\rd )$, 
\[
	Z_t (f) := 
	\ds\sum _{u \in Z_t} f \left( \bfX _t ^u \right), 
	\quad t \ge 0 .
\]
In particular, for $\kappa > 0$, 
\[
	Z_t ^\kappa 
	:= 
	\{ u \in Z_t \; : \; |\bfX _t^u| \ge \kappa \} , 
	\qquad 
	N_t ^\kappa 
	:= 
	\ds\sum _{u \in Z_t} 
	\indi _{[\kappa , \infty)} \left( \left| \bfX _t ^u \right| \right), 
	\qquad t \ge 0 . 
\]
The random variable $N_t^\kappa$ is the number of particles 
which stay outside $B(R)$, 
where $B(R)$ is a sphere of radius $R$ centered at the origin.

Let us set 
\[
	Q(x) = \sum _{n=1} ^{\infty} n p_n (x), \qquad 
	R(x) = \sum _{n = 2} ^{\infty} n (n-1) p_n (x) .
\]
We write $R \mu $ in the meaning of $R(x) \mu (d x)$, 
and we write $(Q-1) \mu$ in the same meaning.  
\begin{assumption}\label{ass1}
For the branching rate measure $\mu \in \calK$, 
we assume the following: 

\begin{enumerate}
\item[{\rm (i)}] The support of $\mu$ is compact.

\item[{\rm (ii)}] $R(x) \mu (d x) \in \calK$. 

\item[{\rm (iii)}] $\lambda ((Q-1) \mu) < 0$. 

\end{enumerate}
\end{assumption}
By Assumption \ref{ass1} (iii) and Proposition \ref{prop4}, 
$\lambda ( (Q-1) \mu)$ is the eigenvalue. 
The corresponding eigenfunction $h$ is bounded and 
$L^2$-normalized strictly positive continuous function. 
For $h$, we define a martingale:  
\begin{equation}
	M_t = e^{\lambda ((Q-1)\mu)t} Z_t (h), \quad t \ge 0 .
	\label{mart} 
\end{equation} 
By \cite[Lemma 3.4]{S08}, $M_t$ is square integrable. 
Thus, $M_\infty := \lim _{t \to \infty} M_t$ exists in $[0, \infty)$, $\prob_x$-a.s 
and	$h(x) = \Ex _x \left[ M_\infty \right]$ (e.g. \cite[Theorem 4.4.6]{Durrett}). 
Hence, $M_\infty < \infty$, $\prob _x$-a.s. 
and $\prob _x (M_\infty > 0 ) > 0$. 

We introduce Many-to-One and historical Many-to-One lemmas. 
\begin{theorem}[Lemma 3.3 in \cite{S08}]\label{manytoone}
If $\sup _{x \in \rd} Q(x) < \infty$, then for any $f \in \calB _b (\rd)$, 
\[
	\Ex_x \left[
		Z_t (f)
	\right]
	= 
	E_x \left[e ^{A_t^{(Q-1) \mu}} f (X_t) \right] .
\]
\end{theorem}
Similar to \cite[Lemma 3.6]{S18} for the Brownian case, 
we have the historical type Many-to-One lemma below. 
For $u \in Z_t$ and its ancestor $u_0 \in Z_0$, 
the genealogy is unique. 
We write $\{ \bfX ^{(t,u)} _s \}_{0 \le s \le t}$ 
for the historical path between $\bfX _t ^u$ and $\bfX_0 ^{u_0} (= x)$. 
Then we are able to regard the historical path 
as the symmetric $\alpha$-stable path. 

\begin{lemma}\label{lem:histrical-many-to-one}
Let $\mu \in \calK$. 
For any $x \in \rd$, $t > 0$ and $\kappa \ge 0$, 
\[
	\Ex _x \left[
	\sum _{u \in Z_t} \indi _{[\kappa , \infty)} 
	\left( \sup _{0 \le s \le t} \left| \bfX _s ^{(t, u)} \right| \right)
	\right]
	= 
	E_x \left[
		e^{A_t^{(Q-1)\mu}} ; \sup _{0 \le s \le t} |X_s| \ge \kappa
	\right] .
\]
\end{lemma}

\section{Proof of Theorem \ref{theo:main}}
In this section, 
we always assume that the branching rate measure $\mu$ satisfies 
Assumption \ref{ass1}. 
We set $\nu = (Q-1) \mu$ and 
$\lambda = \lambda (\nu) < 0$. 
For $\lambda$ and the corresponding eigenfunction $h$, 
we define $M_t$ by \eqref{mart}. 
Additionally, the limit $M_\infty$ exists in $[0, \infty)$, $\prob_x$-a.s. 
and $\prob_x (M_\infty > 0) > 0$. 

Shiozawa proved the same claim for the Brownian case as 
Lemmas \ref{lem:3-7}--\ref{lem:liminf} below 
(Lemma 3.7--3.9 in \cite{S18}). 
The $\alpha$-stable cases are also proved by his method. 
In the proofs of Lemmas \ref{lem:limsup} and \ref{lem:liminf},  
a change is that 
we use the running maximum of the $\alpha$-stable process. 

In our argument, letter `$x$'  is the fixed starting point. 
Other letters as starting points are variable. 
\begin{lemma}[cf. Lemma 3.7 in \cite{S18}]\label{lem:3-7}
For any $x \in \rd$, 
\[
	\lim_{t \to \infty} \frac{1}{t} \log N_t 
	= 
	\lim_{t \to \infty} \frac{1}{t} 
	\log \Ex _x [N_t ] 
	= 
	-\lambda , \qquad 
	\prob _x (\cdot \mid M_{\infty} > 0)\text{-a.s.}
\]
\end{lemma}

\begin{proof}
Theorem \ref{manytoone} leads to 
$\Ex _x [N_t] = E_x \left[ e^{A_t^\nu} \right]$. 
By \eqref{eq:3-11}, we have the second equation.  

Since $h$ is bounded, 
\[
	M_t = e^{\lambda t} Z_t (h) \le e^{\lambda t} 
	\| h \| _{\infty} N_t . 
\]
Because $M_\infty < \infty$, $\prob_x$-a.s., 
we have
\begin{eqnarray}
	\liminf _{t \to \infty} \frac{1}{t} 
		\log N_t 
	&\ge & 
	\liminf_{t \to \infty} \frac{1}{t} \log 
		\left(
			e^{- \lambda t} \| h \| _{\infty} ^{-1} M_t
		\right)
	\nonumber 
	\\
	&=& 
	- \lambda 
	+ 
	\liminf _{t \to \infty} \frac{1}{t} \log M_t 
	= 
	-\lambda ,
	\label{eq:20230210-2}
\end{eqnarray}
$\prob_x (\cdot \mid M_\infty > 0)$-a.s.

By \eqref{eq:3-11}, 
for any $\varepsilon >0$, 
there exist positive constants $c$ and $T$ such that 
$E_x [e^{A_t ^\nu}] \le c e^{-\lambda t}$, $t > T$. 
By Chebyshev's inequality, 
\[
	\prob _x \left(
		e^{(\lambda - \varepsilon ) t} N_t > \varepsilon
	\right)
	\le 
	\frac{e^{(\lambda - \varepsilon ) t}}{\varepsilon} 
	\Ex _x [N_t] 
	=
	\frac{e^{(\lambda - \varepsilon ) t}}{\varepsilon} 
	E_x \left[
		e^{A_t ^{\nu}}
	\right] 
	\le 
	\frac{c}{\varepsilon} e^{- \varepsilon t} , 
	\quad t > T .
\]
Thus, for $N >T$, 
\[
	\sum _{n=1} ^{\infty} 
	\prob _x \left(
		e^{(\lambda - \varepsilon) n} N_n > \varepsilon 
	\right)
	\le 
	N
	+ 
	\frac{c}{\varepsilon} 
	\sum_{n=N+1} ^{\infty} e^{- \varepsilon n} 
	< \infty . 
\] 
The Borel-Cantelli lemma yields 
\[ 
	\prob_x \left(
		\bigcap _{n=1} ^{\infty} 
		\bigcup _{k=n} ^{\infty}
		\left\{ 
			e^{(\lambda - \varepsilon)k} N_k > \varepsilon 
		\right\}
	\right)
	= 0 .
\] 
Namely, for almost all $\omega$, 
there exists $n = n(\omega) \in \n$ such that 
\[
	e^{(\lambda - \varepsilon) k} N_k (\omega ) \le \varepsilon, 
	\quad \text{for all } k \ge n .
\]
Since $p_0 \equiv 0$, i.e., $N_t$ is nondecreasing, 
for any $t>n$, 
\[
	N_t (\omega) \le N_{[t]+1} (\omega)
	\le 
	\varepsilon e^{-(\lambda - \varepsilon) ([t] + 1)} 
	\le 
	\varepsilon e^{(- \lambda + \varepsilon)(t+1)} , 
\]
where $[t]$ is the greatest integer such that $[t] \le t$. 
From this, it follows that 
\[
	\limsup _{t \to \infty} \frac{1}{t} \log N_t (\omega)
	\le 
	- \lambda + \varepsilon , 
	\qquad \prob _x \text{-a.a. } \omega .
\] 
In a general manner, 
letting $\varepsilon \downarrow 0$, we obtain 
\begin{equation}
	\prob _x \left( 
		\limsup _{t \to \infty} 
		\frac{1}t \log N_t  \le - \lambda 
	\right) 
	= 1.
	\label{eq:20230210-3}
\end{equation}

From \eqref{eq:20230210-2} and \eqref{eq:20230210-3}, 
we come to the conclusion. 
\end{proof} 

\begin{lemma}\label{lem:limsup}
For $\delta > 0$, 
we set 
$\kappa _\delta (t) = e^{\delta t} a(t)$, 
where $a(t) > 0$ is monotone increasing 
and $t^{-1} \log a(t) \to 0$, $t \to \infty$. 
Then, for any $x \in \rd$, 
\[
	\limsup _{t \to \infty} 
	\frac{1}{t} \log N_t ^{\kappa _\delta (t)}
	\le 
	-\lambda - \alpha \delta 
	, 
	\qquad 
	\prob _x \text{-a.s.}
\]
\end{lemma}

\begin{proof}
For $u \in Z_t$, 
Let us denote by 
$\{ \bfX ^{(t,u)} _s \}_{0 \le s \le t}$ 
the historical path connecting $\bfX _t ^u$ and $\bfX_0 = x$. 
For given $t>0$, 
we choose $i \in \n $ with $i \le t < i+1$. 
Then, 
\begin{equation}
	N_t ^{\kappa } 
	\le 
	\sum _{u \in Z_{i+1}} 
	\indi _{[\kappa , \infty)}
	\left( 
		\ds\sup _{i \le s \le i+1} 
		\left| \bfX _s ^{(i+1 , u)} \right|
	\right) , 
	\qquad \kappa > 0 .
	\label{eq:3-25}
\end{equation}
For $\varepsilon > 0$, 
\[
	A _i = A_i (\varepsilon) := 
	\left\{ 
		\sum_{u \in Z_{i+1}} 
			\indi _{[\kappa _\delta (i), \infty)}
			\left(
			\ds\sup _{i \le s \le i+1} 
				\left|
					\bfX _s ^{(i+1 , u)}
				\right|
			\right)
		> 
		e^{(-\lambda - \alpha \delta + \varepsilon) i}
	\right\}, 
	\qquad i \ge 1 .  
\]
It is sufficient to show that 
there exist $c_0 > 0$ and $I_0 \in \n$ such that 
\begin{equation}
	\prob _x \left( A_i \right)
	\le 
	e^{- c_0 i}, 
	\qquad i \ge I_0.
	\label{eq:3-26}
\end{equation}
Indeed, if \eqref{eq:3-26} holds, then 
\[
	\sum _{i=1} ^{\infty} 
	\prob _x (A_i) 
	\le 
	I_0 + \sum_{i = I_0+1} ^{\infty} e^{-c_0 i}
	< \infty , 
\]
which implies 
\[
	\prob _x \left(
		\bigcup _{n=1} ^{\infty} 
		\bigcap _{i=n} ^{\infty}
		A_i ^c
	\right)
	= 1 , 
\] 
by the Borel-Cantelli lemma. 
That is, 
for almost all $\omega$, 
there exists $I = I (\omega) \in \n$ 
such that $\omega \in A_i ^c$ for all $i \ge I$. 
For $t>I$, we choose $i \in \n$ with $i \le t < i+1$. 
Since $\kappa _\delta (t)$ is monotone increasing, 
\begin{eqnarray*} 
	N_t ^{\kappa _\delta (t)} (\omega )
	& \underset{\eqref{eq:3-25}}{\le} & 
	\sum_{u \in Z_{i+1}} 
		\indi _{[\kappa _\delta (t) , \infty)}
		\left(
		\ds\sup _{i \le s \le i+1} 
		\left|
			\bfX _s ^{(i+1 , u)} (\omega )
		\right|
		\right)
	\\ 
	&\le &
	\sum_{u \in Z_{i+1}} 
		\indi _{[\kappa _\delta (i) , \infty)}
		\left(
		\ds\sup _{i \le s \le i+1} 
		\left|
			\bfX _s ^{(i+1 , u)} (\omega )
		\right|
		\right)
	\\ 
	&\le & 	
	e^{(-\lambda - \alpha\delta + \varepsilon) i} .
\end{eqnarray*}
In general, 
$e^{Ai} \le (1 \vee e^{-A}) e^{At}$, 
for any $A \in \R$ and $i \le t$. 
Hence, for $\prob _x$-a.a. $\omega$, 
$$  
	\limsup _{t \to \infty}
	\dfrac{1}t
	\log N_t ^{\kappa _\delta (t)} (\omega)
	\le 
	\limsup _{t \to \infty}
	\frac{1}t
	\log \left\{ 1 \vee e^{-(-\lambda - \alpha \delta + \varepsilon)} \right\}
	e^{(-\lambda - \alpha\delta + \varepsilon)t}
	= 
	-\lambda - \alpha \delta + \varepsilon . 
$$
Since the above holds for all $\varepsilon > 0$, we have  
$$
	\prob _x 
	\left( 
		\limsup _{t \to \infty} 
		\frac{1}t \log Z_t ^{\kappa _\delta (t)}
		\le - \lambda - \alpha \delta
	\right) 
	= 1 .
$$

From now on, we show the existence of 
$c_0>0$ and $I _0 \in \n$ which satisfy \eqref{eq:3-26}, 
for each $\varepsilon > 0$. 
By Chebyshev's inequality and the Markov property, 
\begin{eqnarray}
	\prob _x (A_i) 
	&\le &
	e ^{(\lambda + \alpha \delta - \varepsilon)i} 
	\Ex _x \left[
		\sum_{u \in Z_{i+1}} 
			\indi _{[\kappa _\delta (i) , \infty )}
				\left(
					\ds\sup _{i \le s \le i+1} 
					\left| \bfX _s ^{(i+1 , u)} \right|
				\right) 
	\right]
	\nonumber
	\\ 
	&=&
	e ^{(\lambda + \alpha \delta - \varepsilon)i} 
	\Ex _x  
	\left[
		\Ex _{\bfX _i}
		\left[
			\sum_{u \in Z_{1}} 
			\indi _{[\kappa _\delta(i), \infty)} 
			\left(
				\ds\sup _{0 \le s \le 1} 
				\left| \bfX _s ^{(1 , u)} \right|
			\right)
		\right]
	\right] 
	\nonumber
	\\ 
	&=&
	e ^{(\lambda + \alpha \delta - \varepsilon)i} 
	\Ex _x  
	\left[
		\sum _{v \in Z_i}
			\Ex _{\bfX _i ^v}
			\left[
			\sum_{u \in Z_1} 
			\indi _{[\kappa _\delta (i), \infty)} 
			\left(
				\ds\sup _{0 \le s \le 1} 
				\left| \bfX _s ^{(1 , u)} \right|
			\right)
			\right]
	\right]	.
	\label{eq:20180214-1}
\end{eqnarray}
By Lemma \ref{lem:histrical-many-to-one}, 
\[
	\Ex _y
	\left[
		\sum_{u \in Z_1} 
			\indi _{[\kappa _\delta (i), \infty)} 
			\left(
				\ds\sup _{0 \le s \le 1} 
				\left| \bfX _s ^{(1 , u)} \right|
			\right)
	\right]	
	= 
	E_y
	\left[
		e^{A_1 ^\nu} 
		\; ; \; 
		\sup _{0 \le s \le 1} |X_s| \ge \kappa _\delta (i)
	\right] 
	=: 
	f \left( y \right).
\]
For the above $f \in \calB _b (\rd)$, 
we use Theorem \ref{manytoone} to \eqref{eq:20180214-1}. 
Thus, for any $i \ge 1$, 
\begin{eqnarray*}
	\prob _x (A_i)
	&\le & 
	e ^{(\lambda + \alpha \delta - \varepsilon)i} 
	\Ex _x  
	\left[
		\sum _{v \in Z_i}
		f \left( \bfX _i ^v \right) 
	\right]
	= 
	e ^{(\lambda + \alpha \delta - \varepsilon)i} 
	E _x \left[ e^{A_i ^\nu} f (X_i) \right] 
	\\
	&=&
	e ^{(\lambda + \alpha \delta - \varepsilon)i} 
	E _x \left[ 
		e^{A_i ^\nu}  
		E_{X _i} 
		\left[
			e^{A_1 ^\nu} \; ; \; 
			\sup _{0 \le s \le 1} |X_s| \ge \kappa _\delta (i)
		\right] 
	\right]
	\\
	&=& 
	e ^{(\lambda + \alpha \delta - \varepsilon)i} 
	E _x \left[ 
		e^{A_i ^\nu + A_1 ^\nu \circ \theta _i} \; ; \; 
			\sup _{0 \le s \le 1} |X_{i+s}| \ge \kappa _\delta (i)
	\right]	 
	\\
	&=& 
	e ^{(\lambda + \alpha \delta - \varepsilon)i} 
	E _x \left[ 
		e^{A_{i+1} ^\nu} \; ; \; 
			\sup _{i \le s \le i+1} |X_s| \ge \kappa _\delta (i)
	\right]	 .
\end{eqnarray*}
Let $\eta \in (0,\delta)$. 
We divide the above expectation into (I)$+$(II) as follows:  
\begin{eqnarray*}
	&& 
	E _x \left[ 
		e^{A_{i+1} ^\nu} \; ; \; 
			\sup _{i \le s \le i+1} |X_s| \ge \kappa _\delta (i)
	\right] 
	\\ 
	&=&
	E _x \left[ 
			e^{A_{i+1} ^\nu} \; ; \; 
			\sup _{i \le s \le i+1} |X_s| \ge \kappa _\delta (i), \;
			|X_{i+1}| \ge \kappa _\eta (i+1)
	\right] 
	\\ 
	&& \quad 
	+
	E _x \left[ 
			e^{A_{i+1} ^\nu} \; ; \; 
			\sup _{i \le s \le i+1} |X_s| \ge \kappa _\delta (i), \;
			|X_{i+1}| < \kappa _\eta (i+1)
	\right] 
	\\
	&=:& 
	\text{(I)} + \text{(II)} .
\end{eqnarray*}
By \eqref{eq:3-1}, 
for any $\varepsilon>0$ and $\eta \in (0, \delta)$, 
there exists $T_1 >0$ ($T_1=T_1(\eta , \varepsilon)$) 
such that 
\begin{equation}
	E_x \left[ e^{A_t ^\nu} \; ; \; |X_t| > \kappa _\eta (t) \right]
	< 
	e^{(-\lambda -\alpha \eta) t} , \qquad 
	t > T_1 .
	\label{eq:20230509-1}
\end{equation}
Here, we choose the same $\varepsilon$ as in $A_i (\varepsilon)$ and fix it, 
and then, we consider $T_1 = T_1 (\eta)$. 
When $i > T_1$, 
\begin{equation} 
	\text{(I)}
	\le 
	E _x \left[ 
		e^{A_{i+1} ^\nu} \; ; \; 
		|X_{i+1}| \ge \kappa _\eta (i+1)
	\right]	
	\le
	e^{(-\lambda -\alpha \eta) (i+1)} .
	\label{eq:PartI}
\end{equation}
We divide (II) into (III)$+$(IV) as follows: 
\begin{eqnarray*}
	\text{(II)} 
	&=& 
	E _x \left[ 
		e^{A_{i+1} ^\nu} 
		\; ; \; 
		\sup _{i \le s \le i+1} |X_s| \ge \kappa _\delta (i), \;
		|X_i| \ge \kappa _\eta (i), \; 
		|X_{i+1}| < \kappa _\eta (i+1)
	\right] 
	\\
	&& \quad + 
	E _x \left[ 
		e^{A_{i+1} ^\nu} 
		\; ; \; 
		\sup _{i \le s \le i+1} |X_s| \ge \kappa _\delta (i), \;
		|X_i| < \kappa _\eta (i), \; 
		|X_{i+1}| < \kappa _\eta (i+1)
	\right] 	
	\\ 
	&\le &
	E _x \left[ 
		e^{A_{i+1} ^\nu} 
		\; ; \; 
		|X_i| \ge \kappa _\eta (i)
	\right] 
	+
	E _x \left[ 
		e^{A_{i+1} ^\nu} 
		\; ; \; 
		\sup _{i \le s \le i+1} |X_s| \ge \kappa _\delta (i), \;
		|X_i| < \kappa _\eta (i) 
	\right] 
	\\
	&=:&
	\text{(III)}+ \text{(IV)}.
\end{eqnarray*}
By the Markov property, 
\[
	\text{(III)}
	= 
	E _x \left[
		e^{A_i ^\nu} 
		E_{X_i} \left[
			e^{A_1 ^\nu} 
		\right]
		\; ; \; 
		|X_i| \ge \kappa _\eta (i)
	\right] 
	\le
	\left( 
		\sup _{y \in \rd}E_y \left[ e^{A_1 ^\nu} \right]
	\right) 
	E _x \left[
		e^{A_i ^\nu} 
		\; ; \; 
		|X_i| \ge \kappa _\eta (i)
	\right] . 
\] 
\cite[Theorem 6.1]{ABM91} leads to 
$\sup _{y \in \rd}E_y [ e^{A_1 ^\nu} ] < \infty$. 
Thus, by \eqref{eq:20230509-1},  
\begin{equation}
	\text{(III)}
	\le 
	c e^{(-\lambda -\alpha \eta) i} , 
	\qquad i > T_1 .
	\label{eq:PartIII}
\end{equation}
Then, by the Markov property, 
\begin{equation}
	\text{(IV)}
	=
	E _x \left[ 
		e^{A_i ^\nu} 
		E_{X_i} \left[ 
			e^{A_1 ^\nu}
			\; ; \; 
			\sup _{0 \le s \le 1} |X_s| \ge \kappa _\delta (i)
		\right] 
		\; ; \; 
		|X_i| < \kappa _\eta (i)
	\right] .
	\label{eq:20230425-2}
\end{equation}
By H\"older's inequality, 
for any $\theta > 1$, 
\begin{equation} 
	E_y \left[ 
		e^{A_{1} ^\nu} 
		\; ; \; 
		\sup _{0 \le s \le 1} |X_s| \ge \kappa _\delta (i)
		\right] 
	\le 
	E_y \left[ 
		e^{ \frac{\theta}{\theta -1} A_1 ^\nu } 
	\right]^{1 -1/\theta }  
	P_y \left( 
		\sup _{0 \le s \le 1} |X_s| \ge \kappa _\delta (i)
	\right) ^{1/\theta} , 
	\label{eq:20230425-1}
\end{equation}
where $|y| \le \kappa _\eta (i)$ $(<\kappa _\delta (i))$. 
On the right side, the expectation is bounded above by some constant $C = C_\theta$, 
due to \cite[Theorem 6.1]{ABM91}. 
By Lemma \ref{lem:tailmax}, 
there exist $C'>0$ and $T_2 > T_1$ such that, 
for any $i \ge T_2$ and $|y| \le \kappa _\eta (i)$, 
\[
	P_y \left( 
		\sup _{0 \le s \le 1} |X_s| \ge \kappa _\delta (i)
	\right) 
	\le 
	C' \left( \kappa _\delta (i) - \kappa _\eta (i) \right) ^{-\alpha} .
\] 
For $i \ge T_2$, we have by \eqref{eq:20230425-1}, 
\[
	\sup _{|y| < \kappa _\eta (i)} 
	E_y \left[ 
		e^{A_{1} ^\nu} 
		\; ; \; 
		\sup _{0 \le s \le 1} |X_s| \ge \kappa _\delta (i)
	\right] 
	\le 
	C \left( \kappa _\delta (i) - \kappa _\eta (i) \right) ^{-\alpha / \theta }. 
\]
Hence, by \eqref{eq:20230425-2}, 
\begin{eqnarray*}
	\text{(IV)} 
	&\le &
	C \left( \kappa _\delta (i) - \kappa _\eta (i) \right) ^{-\alpha / \theta }
	E _x \left[ 
		e^{A_i ^\nu} 
		\; ; \; 
		|X_i| < \kappa _\eta (i)
	\right] 
	\\
	&\le &
	C \left( \kappa _\delta (i) - \kappa _\eta (i) \right) ^{-\alpha / \theta }
	E _x \left[ e^{A_i ^\nu} \right] , 
\end{eqnarray*}
for $i \ge T_2$. 
According to \eqref{eq:3-11}, we choose $T_3 \ge T_2$ such that 
$E_x [e^{A_t ^\nu}] \le c e ^{-\lambda t}$, $t \ge T_3$. 
Therefore, for $i \ge T_3$, 
\begin{eqnarray}
	\text{(IV)} 
	&\le &
	C \left( \kappa _\delta (i) - \kappa _\eta (i) \right) ^{-\alpha / \theta } 
	e^{-\lambda i}
	\nonumber
	\\
	&=&
	C \left\{
		e ^{\delta i} a(i) \left( 1 - e^{-(\delta -\eta) i} \right)
	\right\}^{-\alpha / \theta } 
	e^{-\lambda i}
	\nonumber
	\\
	&\le &
	C' e ^{(- \alpha \delta /\theta - \lambda )i}  , 
	\label{eq:PartIV}
\end{eqnarray}
because $1 - e ^{-(\delta - \eta)i} \uparrow 1$ and 
$a(i)$ is monotone increasing. 

Combining \eqref{eq:PartI}, \eqref{eq:PartIII} and \eqref{eq:PartIV}, 
we are able to choose $I_0 \ge 1$ ($I_0= T_3 (\eta)$) such that, 
for $i \ge I_0$, 
\begin{eqnarray*}
	\prob_x (A_i)
	&\le & 
	e^{(\lambda + \alpha \delta - \varepsilon ) i }
	\left( \text{(I)} + \text{(III)} + \text{(IV)} \right)
	\\ 
	&\le &
	e^{(\lambda + \alpha \delta - \varepsilon ) i }
	\left\{ 
		e^{(-\lambda -\alpha \eta) (i+1)} 
		+
		c e^{(-\lambda -\alpha \eta) i}
		+ 
		C e ^{(- \alpha \delta /\theta - \lambda )i} 
	\right\} 
	\\
	&=&
	e^{(\lambda + \alpha \delta - \varepsilon ) i }
	\left\{ 
		c' e^{(-\lambda -\alpha \eta) i}
		+ 
		C e ^{(- \alpha \delta /\theta - \lambda )i} 
	\right\}
	\\
	&=& 
	e ^{\{ \alpha (\delta - \eta) - \varepsilon\} i}
	\left[
		c' + C e^{\alpha (\eta - \delta / \theta) i }
	\right] .
\end{eqnarray*} 
For fixed any $\alpha , \delta$ and $\varepsilon$, 
we choose $\theta > 1$ and $\eta < \delta $ such that
\[ 
	\alpha (\delta - \eta ) - \varepsilon< 0 ,
	\qquad 
	\alpha \left(\eta - \frac{\delta}{\theta} \right) 
	< 0 . 
\] 
For each $\varepsilon > 0$, 
we prove the existence of $c_0, I_0$ which satisfy \eqref{eq:3-26}. 
\end{proof}

\begin{lemma}\label{lem:liminf}
For $\delta > 0$, 
we set 
$\kappa _\delta (t) = e^{\delta t} a(t)$,
where $a(t) > 0$ is monotone increasing 
and $t^{-1} \log a(t) \to 0$, $t \to \infty$. 
If $\delta < - \lambda / \alpha$, 
then for any $x \in \rd $, 
\[
	\liminf _{t \to + \infty} 
	\frac{1}{t} \log N_t ^{\kappa _\delta (t)}
	\ge 
	- \lambda - \alpha \delta , 
	\qquad 
	\prob _x (\cdot \mid M_\infty > 0) \text{-a.s.}
\]
\end{lemma}

We set 
\[
	Z_t ([0, \kappa])
	= 
	\left\{
		u \in Z_t \; : \; 
		\left| \bfX _t ^u \right| \in [0, \kappa ]
	\right\}, 
	\qquad \kappa > 0 
\]
and $N_t ([0, \kappa] )$ for the cardinal number of 
$Z_t ([0, \kappa])$.

\begin{proof}
Let us denote by $\{\bfX _s ^{(t,u)} \} _{s \ge t}$ 
the one-particle trajectory which is rooted from $u \in Z_t$. 
The path of $\{\bfX _s ^{(t,u)} \} _{s \ge t}$ is obtained as follows. 
When particle $u$ splits into some particles, 
we choose $v$ from these particles. 
Then, we attach the trajectory of $v$ to one of $u$. 
By repeating this procedure, $\{\bfX _s ^{(t,u)} \} _{s \ge t}$ is constructed. 
This process is regarded as a symmetric $\alpha$-stable process.

We fix $\varepsilon > 0$ and $p \in (0,1)$. 
Then, we shall show that there exist $C >0$ and $N \in \n$ such that 
\begin{equation}
	\prob _x \left(
		\sum _{u \in Z_{n p}} \indi _{E_n ^u} 
		\le 
		e^{(-\lambda - \alpha \delta - \varepsilon) n} , 
		\ 
		N_{n p} \ge e^{-\lambda p^2 n}
	\right) 
	\le 
	e^{-C n}, 
	\qquad n \ge N ,
	\label{eq:20230306-4}
\end{equation}
where, for each $u \in Z_{np}$, the event $E_n^u$ is defined by 
\[
	E_n ^u 
	:= 
	\left\{ 
		\left| \bfX _n ^{(n p , u)} \right| 
		> 
		\left|
			\bfX _n ^{(n p , u)} - \bfX _{n p} ^{(n p , u)} 
		\right| 
		> 
		2 \kappa _\delta (n+1) , \ 
		\sup _{n \le s \le n+1} 
		\left|
			\bfX _s ^{(n p , u)} - \bfX _n ^{(n p , u)}
		\right| 
		< \kappa _\delta (n)
	\right\} .
\] 
From 
$
	\sum _{u \in Z_{np}([0, \kappa])} \indi _{E_n^u} 
	\le 
	\sum _{u \in Z_{np}} \indi _{E_n^u}
$ 
and 
$N_{np} = N_{np} ([0,\kappa]) + N_{np} ^\kappa$ for any $\kappa > 0$, 
we see that for any $\kappa , \gamma >0$, 
the left-hand side of \eqref{eq:20230306-4} is 
\begin{eqnarray*}
	&& 
	\prob _x \left(
		\sum _{u \in Z_{np} ([0, \kappa]) } \indi _{E_n^u} 
		\le 
		e^{(-\lambda - \alpha \delta - \varepsilon) n} , 
		\ 
		N_{n p} ([0, \kappa]) 
		\ge 
		e^{-\lambda p^2 n} - N _{np} ^{\kappa} , 
		\ 
		N_{np} ^{\kappa} \le \gamma
	\right) 
	\\
	&&
	+
	\prob _x \left(
		\sum _{u \in Z_{np} ([0, \kappa]) } \indi _{E_n^u} 
		\le 
		e^{(-\lambda - \alpha \delta - \varepsilon) n} , 
		\ 
		N_{n p} ([0, \kappa]) 
		\ge 
		e^{-\lambda p^2 n} - N _{np} ^{\kappa} , 
		\ 
		N_{np} ^{\kappa} > \gamma
	\right) 
	\\
	&\le &
	\prob _x \left(
		\sum _{u \in Z_{np} ([0, \kappa]) } \indi _{E_n^u} 
		\le 
		e^{(-\lambda - \alpha \delta - \varepsilon) n} , 
		\ 
		N_{n p} ([0, \kappa]) 
		\ge 
		e^{-\lambda p^2 n} - \gamma
	\right) 
	+
	\prob _x \left(
		N_{np} ^{\kappa} > \gamma
	\right) . 
\end{eqnarray*}
Instead of $\kappa$ and $\gamma$, 
we set $\kappa _n = e ^{\eta _1 n} a(n)$ and $\gamma _n = e ^{\eta _2 n} a(n)$ 
for $0<\eta _2 < \eta _1 < \delta$. 
Thus, it is sufficient for \eqref{eq:20230306-4} 
to show that 
there exist positive constants $c_i$ ($i=1,2$) 
and $N \in \n$ such that, for all $n \ge N$,   
\begin{equation}
	\begin{aligned}
	\text{(I)} &:= 
	\prob _x \left(
		\sum _{u \in Z_{np} ([0, \kappa_n]) } \indi _{E_n^u} 
		\le 
		e^{(-\lambda - \alpha \delta - \varepsilon) n} , 
		\ 
		N_{n p} ([0, \kappa_n]) 
		\ge 
		e^{-\lambda p^2 n} - \gamma _n
	\right) 
	\le 
	e^{-c_1 n} , \\
	\text{(II)} &:= 
	\prob _x \left(
		N_{np} ^{\kappa_n} > \gamma _n
	\right) 
	\le e^{-c_2 n} .
	\end{aligned}
	\label{eq:aimII}
\end{equation} 
Our task is to give appropriate $p \in (0,1)$ and $\eta _i$ ($i=1,2$), 
which ensure the existence of $c_i$ and $N$ for \eqref{eq:aimII}. 
We determine the requirements from (I) and (II). 
As a result, such conditions will appear in \eqref{condition3} below. 

We firstly consider a condition for (II). 
By Chebyshev's inequality and \eqref{eq:3-1}, 
there exist $c>0$ and $N_1 \in \n$ such that for $n \ge N_1$,   
\begin{eqnarray}
	\text{(II)}
	&\le&  
	\left( \gamma_n \right) ^{-1} 
	\Ex _x \left[ N_{np} ^{\kappa _n} \right]
	\le 
	c \left( \gamma _n \right) ^{-1} 
	\left( \kappa _n \right)^{-\alpha} 
	e^{-\lambda np} 
	\nonumber
	\\
	&= &
	\dfrac{c}{a(n)^{1+\alpha}} 
	e^{(-\alpha \eta _1 - \eta _2 - \lambda p)n}
	\le 
	e^{(-\alpha \eta _1 - \eta _2 - \lambda p)n} .
	\label{condition1}
\end{eqnarray}

We secondly consider a condition for (I). 
By Chebyshev's inequality, 
\begin{eqnarray}
	&&
	\text{(I)}
	=
	\prob _x \left(
		\sum _{u \in Z_{n p} ([0, \kappa _n]) } 
		\indi _{E_n ^u} 
		\le 
		e^{(-\lambda - \alpha \delta - \varepsilon) n} , 
		\ 
		N_{n p} ( [0, \kappa _n] ) 
		\ge 
		e^{-\lambda p^2 n} - \gamma _n
	\right)
	\nonumber
	\\ 
	&= & 
	\prob _x \left(
		\exp \left(
			- \sum _{u \in Z_{n p} ( [0, \kappa _n] ) } 
			\indi _{E_n ^u} 
		\right) 
		\ge 
		\exp \left\{ - e^{(-\lambda - \alpha \delta - \varepsilon) n} \right\} , 
		\ 
		N_{n p} ([0, \kappa_n]) 
		\ge 
		e^{-\lambda p^2 n} - \gamma _n
	\right)
	\nonumber 
	\\ 
	&\le & 
	\exp\left\{ e^{(-\lambda - \alpha \delta - \varepsilon) n} \right\}
	\Ex _x \left[
		\exp \left(
			- \sum _{u \in Z_{n p} ( [0, \kappa_n] )}
			\indi _{E_n ^u} 
		\right)
		\; ; \;  
		N_{n p} ( [0, \kappa_n] )
		\ge 
		e^{-\lambda p^2 n} - \gamma _n
	\right] 
	\nonumber
	\\
	&=&
	\exp\left\{ e^{(- \lambda - \alpha \delta - \varepsilon) n} \right\}
	\Ex _x \left[
		\prod _{u \in Z_{np} ( [0, \kappa_n] ) }
		e^{-\indi _{E_n ^u}} 
		\; ; \;  
		N_{n p} ( [0, \kappa_n] ) 
		\ge 
		e^{-\lambda p^2 n} - \gamma _n
	\right] .
	\label{eq:20230216-1}
\end{eqnarray} 
We write 
$$ 
	F_n ^u 
	=
	\left\{ 
		\left| \bfX _{n-np} ^{(0 , u)} \right| 
		> 
		\left|
			\bfX _{n-np} ^{(0 , u)} - \bfX _0 ^{(0 , u)} 
		\right| 
		> 
		2 \kappa _\delta (n+1) , \, 
	\sup _{n - n p \le s \le n+1 - n p} 
		\left|
			\bfX _s ^{(0,u)} - \bfX _{n - n p} ^{(0 , u)}
		\right| 
		< \kappa _\delta (n)
	\right\} .
$$ 
Since 
$\{ \{ \bfX ^{(np,u)} _s \}_{s \ge np}, u \in Z_{np} \}$ 
are mutually independent under the law $\prob _x (\cdot \mid \calF _{np})$, 
the second factor of \eqref{eq:20230216-1} is equal to 
\begin{eqnarray}
	&&
	\Ex _x \left[
	\Ex _x \left[ \left. 
		\prod _{u \in Z_{np} ( [0, \kappa_n] )} 
		e^{-\indi _{E_n ^u}} 
		\; \right| \; \calF _{np}
	\right] 
		\; ; \;  
		N_{n p} ( [0, \kappa_n] ) 
		\ge 
		e^{-\lambda p^2 n} - \gamma _n
	\right] 
	\nonumber 
	\\
	&=& 
	\Ex _x \left[
	\Ex _{\bfX _{np}} \left[ 
		\prod _{u \in Z_0 ( [0, \kappa_n] )} 
		e^{-\indi _{F_n ^u}} 
		\right] 
		\; ; \;   
		N_{n p} ( [0, \kappa_n] ) 
		\ge 
		e^{-\lambda p^2 n} - \gamma _n
	\right]
	\nonumber
	\\
	&=&
	\Ex _x \left[ 
		\prod _{u \in Z_{np} ( [0, \kappa _n] )}
		\Ex _{\bfX_{np} ^u} \left[  
			e^{-\indi _{F_n ^u}} 
		\right] 
		\; ; \;   
		N_{n p} ( [0, \kappa _n] )
		\ge 
		e^{-\lambda p^2 n} - \gamma _n 
	\right] . 
	\label{eq:20230218-3}
\end{eqnarray}
Here, for any $u \in Z_{np} ( [0, \kappa_n] )$,  
\begin{equation}
	\Ex _{{\bf X}_{np} ^u} \left[  
		e^{-\indi _{F_n ^u}} 
	\right] 
	= 
	1 - (1-e^{-1}) \prob _{{\bf X}_{np} ^u} (F_n^u).	
	\label{eq:20230216-2} 
\end{equation} 
We give the lower estimate of $\prob_{{\bf X}_{np} ^u} (F_n^u)$ 
for $u \in Z_{np} ( [0, \kappa _n] )$, 
which implies the upper estimate of \eqref{eq:20230218-3}. 
Since $\{ \bfX_s ^{(0,u)} \}_{s \ge 0}$ 
is identified with the stable process, 
for $|w| \le \kappa _n$, 
\begin{eqnarray}
	\prob _w (F_n^u)
	&=&
	P_w 
	\bigg( 
		\left| X _{n-np} \right| 
		> 
		\left|
			X _{n-np} - X _0 
		\right| 
		> 
		2 \kappa _\delta (n+1), 
	\nonumber 
	\\ 
	&& \quad 
		\sup _{n - n p \le s \le n+1 - n p} 
		\left|
			X _s - X _{n - n p}
		\right| 
		< \kappa _\delta (n)
	\bigg) 
	\nonumber 
	\\
	&=& 
	P_w
	\left( 
		\left| X _{n-np} \right| 
		> 
		\left| X _{n-np} - w \right| 
		> 
		2 \kappa _\delta (n+1)
	\right) 
	\nonumber 
	\\ 
	&& \times 
	P_w 
	\left(
	\sup _{n - n p \le s \le n+1 - n p} 
		\left| X _s  - X _{n - n p} \right| 
		< 
		\kappa _\delta (n)
	\right) . 
	\label{eq:20230410-1}
\end{eqnarray}
We fix $\eta \in (\eta _1 , \delta )$ and 
set $\kappa_\eta (n) = e^{\eta n} a(n)$. 
Then,  
\begin{eqnarray}
	&&
	P_w 
	\left(
	\sup _{n - n p \le s \le n+1 - n p} 
		\left| X _s  - X _{n - n p} \right| 
		< 
		\kappa _\delta (n)
	\right) 
	\nonumber 
	\\
	&\ge &
	P_w
	\left(
	\sup _{n - n p \le s \le n+1 - n p} 
		\left| X _s  - X _{n - n p} \right| 
		< 
		\kappa _\delta (n) , \; 
		\left| X _{n - n p} \right| < \kappa _\eta (n)
	\right) 
	\nonumber 
	\\ 
	&=& 
	E_w \left[  
		P_{X_{n-np}} \left(
			\sup _{0 \le s \le 1} 
			\left| X _s  - X _0 \right| 
			< 
			\kappa _\delta (n)
		\right) 
		\; ; \; 
		\left| X _{n - n p} \right| < \kappa _\eta (n)
	\right] 
	\nonumber 
	\\
	&=& 
	P_w \left(
		\left| X _{n-np} \right| 
		< 
		\kappa _\eta (n)
	\right) 	
	- 
	E_w \left[  
		P_{X_{n-np}} \left(
			\sup _{0 \le s \le 1} 
			\left| X _s  - X _0 \right| 
			\ge 
			\kappa _\delta (n)
		\right) 
		\; ; \; 
		\left| X _{n - n p} \right| < \kappa _\eta (n)
	\right]. 
	\nonumber \\
	\label{eq:20230425-A}
\end{eqnarray}
On the second term of \eqref{eq:20230425-A}, 
from 
$$
	\left\{
		\sup _{0 \le s \le 1} \left| X_s - X_0 \right| 
			\ge \kappa _\delta (n) 
	\right\}
	\subset 
	\left\{
		\sup _{0 \le s \le 1} |X_s |  
			\ge \kappa _\delta (n) - |X_0|
	\right\}
$$ and Lemma \ref{lem:runningmax}, 
we see that 
\[
\begin{split}
	P_y \left(
		\sup _{0 \le s \le 1} 
		\left| X _s  - X _0 \right| 
		\ge 
		\kappa _\delta (n)
	\right) 
	\le 
	P_y \left(
		\sup _{0 \le s \le 1} |X_s |  
			\ge \kappa _\delta (n) - |y|
	\right) 
	\\ 
	\le 
	2 P_y \left(
		|X_1| \ge \kappa _\delta (n) - |y|
	\right) 
	\le 
	2 \omega _d 
	\int ^\infty _{\kappa _\delta (n) - 2 \kappa _\eta (n)}
		g (r) r^{d-1} d r ,
\end{split}
\] 
for all $|y| < \kappa _\eta (n)$. 
Here, we assume $n$ so large that $\kappa _\delta (n) > 2 \kappa _\eta (n)$. 
Thus, 
\begin{eqnarray*}
	&& 
	E_w \left[  
		P_{X_{n-np}} \left(
			\sup _{0 \le s \le 1} 
			\left| X _s  - X _0 \right| 
			\ge 
			\kappa _\delta (n)
		\right) 
		\; ; \; 
		\left| X _{n - n p} \right| < \kappa _\eta (n)
	\right] 
	\\
	&\le &
	2 \omega _d 
	\left( 
		\int ^\infty _{\kappa _\delta (n) - 2 \kappa _\eta (n)}
		g (r) r^{d-1} d r
	\right) 
	P_w \left(
		\left| X _{n - n p} \right| < \kappa _\eta (n)
	\right)	
\end{eqnarray*}
When $|w| \le \kappa _n$, by \eqref{eq:20230206-2}, 
\begin{eqnarray*}
	&&
	P_w ( |X _{n-np}| \le \kappa _\eta (n) ) 
	= 
	P \left( 
		\left| (n-np)^{1/\alpha} X_1 + w \right| 
		\le 
		\kappa _\eta (n) 
	\right) 
	\\ 
	&\ge &
	P \left( 
		\left| (n-np) ^{1/\alpha} X_1 \right| + \left| w \right| 
		\le 
		\kappa _\eta (n) 
	\right) 
	\ge 
	P \left( 
		(n-np)^{1/\alpha} \left| X_1 \right| 
		\le
		\kappa _\eta (n) - \kappa_n
	\right) .
\end{eqnarray*}
Therefore, by \eqref{eq:20230425-A}, 
\begin{eqnarray}
	&& 
	P_w 
	\left(
	\sup _{n - n p \le s \le n+1 - n p} 
		\left| X _s  - X _{n - n p} \right| 
		< 
		\kappa _\delta (n)
	\right)
	\nonumber 
	\\
	&\ge & 
	P_w \left(
		\left| X _{n - n p} \right| < \kappa _\eta (n)
	\right)	
	-
	2 \omega _d 
	\left( 
		\int ^\infty _{\kappa _\delta (n) - 2 \kappa _\eta (n)}
		g (r) r^{d-1} 
	\, d r
	\right) 
	P_w \left(
		\left| X _{n - n p} \right| < \kappa _\eta (n)
	\right)	
	\nonumber 
	\\
	&=& 
	P_w \left(
		\left| X _{n - n p} \right| < \kappa _\eta (n)
	\right)	
	\left(
		1 - 
		2 \omega _d 
		\int ^\infty _{\kappa _\delta (n) - 2 \kappa _\eta (n)}
		g (r) r^{d-1} 
		\, d r		
	\right) 
	\nonumber
	\\ 
	&\ge &
	P \left( 
		\left| X_1 \right| 
		\le 
		\left( \kappa _\eta (n)- \kappa _n^1 \right) 
		(n-np)^{-1/\alpha}
	\right) 
	\left(
		1 - 
		2 \omega _d 
		\int ^\infty _{\kappa _\delta (n) - 2 \kappa _\eta (n)}
		g (r) r^{d-1}
		\, d r
	\right) =:C_n .
	\nonumber 
	\\ 
	\label{eq:20230502-3} 
\end{eqnarray}
Here, $\kappa_\delta (n) - 2 \kappa_\eta (n) \to \infty$, in addition, 
for $\eta \in (\eta _1 , \delta)$, 
\[
	(\kappa _\eta (n)- \kappa _n) n^{-1/\alpha} 
	= 
	\left(e^{\eta n} - e^{\eta _1 n} \right) a(n) n^{-1/\alpha} 
	\to \infty.  
\]
Hence, $C_n \uparrow 1$ as $n \to \infty$.  
Then, the first factor of \eqref{eq:20230410-1} is equal to
\begin{eqnarray}
	&& 
	P_w 
	\left( 
		\left| X _{n-np} \right| 
		> 
		\left| X _{n-np} - x \right| 
		> 
		2 \kappa _\delta (n+1)
	\right) 
	\nonumber
	\\ 
	&=&
	(n-np)^{-d/\alpha}
	\int _{|y| > |y-w|> 2 \kappa _\delta (n+1) }
	g \left( \dfrac{|y-w|}{(n-np)^{1/\alpha}} \right) 
	\, d y
	\nonumber
	\\
	&=&
	(n-np)^{-d/\alpha}
	\int _{|v+w| > |v| > 2 \kappa _\delta (n+1) }
	g \left( \dfrac{|v|}{(n-np)^{1/\alpha}} \right)
	\, d v	.
	\label{eq:20230216-4}
\end{eqnarray} 
Since $|w+v| > |v| \Leftrightarrow |v|^2 + 2 \langle v,w \rangle >0$, 
$
	\left\{ v \in \rd : \langle v,w \rangle > 0 \right\}
	\subset 
	\left\{ v \in \rd : |w+v| > |v| \right\}
$. 
Thus,
\begin{eqnarray*}
	\eqref{eq:20230216-4}
	&\ge &
	(n-np)^{-d/\alpha}
	\int _{|v|> 2 \kappa _\delta (n+1) , \langle v,w \rangle >0}
	g \left( \dfrac{|v|}{(n-np)^{1/\alpha}} \right) 
	\, d v
	\\
	&=&
	\dfrac{\omega_d}{2}
	(n-np)^{-d/\alpha}
	\int ^\infty _{2 \kappa _\delta (n+1) }
	g \left( \dfrac{r}{(n-np)^{1/\alpha}} \right) 
	r^{d-1} 
	\, d r 
	\\
	&=& 
	\dfrac{\omega_d}{2}
	\int ^{\infty} _{ 2 \kappa _\delta (n+1) (n-np)^{-1/\alpha}}
	g \left( u \right) 
	u^{d-1} 
	\, d u . 
\end{eqnarray*}
Hence, we have that for any $w \in \rd$ and $n \ge 1$, 
\begin{equation}
	P_w 
	\left( 
		\left| X _{n-np} \right| 
		> 
		\left| X _{n-np} - w \right| 
		> 
		2 \kappa _\delta (n+1) 
	\right) 
	\ge 
	\dfrac{\omega_d}{2}
	\int ^{\infty} _{ 2 \kappa _\delta (n+1) (n-np)^{-1/\alpha}}
	g \left( u \right) 
	u^{d-1} 
	\, d u 
	\label{eq:20230306-2}
\end{equation} 
and we see that the right-hand side is independent of $w$. 
Since $\kappa _\delta (n+1) n^{-1/\alpha} \to \infty$, 
integrand $g (u) \asymp u ^{-(\alpha + d)}$, and which implies that  
\begin{equation}
	\begin{aligned}
	\int ^{\infty} _{ 2 \kappa _\delta (n+1) (n-np)^{-1/\alpha}}
	g \left( u \right) 
	u^{d-1} 
	\, d u	
	& \asymp 
	\int ^{\infty} _{ 2 \kappa _\delta (n+1) (n-np)^{-1/\alpha}}
	u^{-(\alpha +d)}
	u^{d-1} 
	\, d u	
	\\ 
	& =
	\frac{(1-p) 2^{-\alpha}}{\alpha} 
	\kappa _\delta (n+1) ^{-\alpha} n .
	\end{aligned}
	\label{eq:20230218-1}
\end{equation}
By \eqref{eq:20230306-2} and \eqref{eq:20230218-1}, 
there exist $N _2 \ge 1$ ($N _2 \ge N_1$) and $C>0$ such that, 
for all $ n \ge N _2$ and $\ w \in \rd$,  
\begin{equation}
	P_w 
	\left( 
		\left| X _{n-np} \right| 
		> 
		\left| X _{n-np} - w \right| 
		> 
		2 \kappa _\delta (n)
	\right) 
	\ge
	C \kappa _\delta (n+1) ^{-\alpha} n. 
	\label{eq:20230218-2}
\end{equation}
By \eqref{eq:20230410-1}, \eqref{eq:20230502-3} 
and \eqref{eq:20230218-2}, 
\[ 
	\prob _w (F_n ^u)
	\ge 
	C_n \kappa _\delta (n+1) ^{-\alpha} n , \quad 
	\text{on } |w| \le \kappa _n , \quad 
	\text{for all } n \ge N_2,  
\]
where $C_n$ is independent of $w$ and $C_n \uparrow C$. 
We apply the above inequality to \eqref{eq:20230216-2}, 
then for all $n \ge N_2$ and $u \in Z_{np} ([0, \kappa _n])$, 
\begin{equation} 
	\Ex _{{\bf X}_{np} ^u} \left[  
		e^{-\indi _{F_n ^u}} 
	\right] 
	=  
	1 - (1-e^{-1}) \prob _{{\bf X}_{np} ^u} (F_n^u) 
	\le 
	1 - c \kappa _\delta (n+1)^{-\alpha} n . 
	\label{eq:20230218-4}
\end{equation}
By substituting \eqref{eq:20230218-4} into \eqref{eq:20230218-3}, 
\begin{eqnarray}
	&& 
	\Ex _x \left[
	\prod _{u \in Z_{np} ([0,\kappa_n])} 
	\Ex _{\bfX _{np}^u} 
	\left[ 
		 e^{-\indi _{F_n ^u}} 
	\right] 
		\; ; \;  
		N_{n p} ([0,\kappa _n]) 
		\ge 
		e^{-\lambda p^2 n} - \gamma _n
	\right] 
	\nonumber 
	\\ 
	&\le &
	\Ex _x \left[
	\left(
		1 - c \kappa _\delta (n+1)^{-\alpha} n	
	\right)^{e^{-\lambda p^2 n}- \gamma _n}
	\; ; \;   
	N_{n p} ([0,\kappa _n]) 
	\ge 
	e^{-\lambda p^2 n} - \gamma _n
	\right] 
	\nonumber
	\\
	&=& 
	\left(
		1 - c \kappa _\delta (n+1)^{-\alpha} n
	\right) ^{e^{-\lambda p^2 n} - \gamma _n}
	\prob _x \left(  
		N_{n p} ( [0,\kappa _n] ) 
		\ge 
		e^{-\lambda p^2 n} - \gamma _n
	\right)	
	\nonumber 
	\\
	&\le &
	\left(
		e^{- c \kappa _\delta (n+1)^{-\alpha} n}
	\right)^{e^{-\lambda p^2 n}- \gamma _n} .
	\label{eq:20230306-1}
\end{eqnarray}
Here, inequality \eqref{eq:20230306-1} is caused by $1 - \kappa \le e^{-\kappa}$, 
$\kappa \in \R$ and $\prob_x ( \cdot) \le 1$. 
Applying \eqref{eq:20230306-1} to \eqref{eq:20230216-1}, 
we have that, for all $n \ge N_2$, 
\begin{eqnarray*}
	\text{(I)}
	&\le &
	\exp\left\{ e^{(-\lambda - \alpha \delta - \varepsilon) n} \right\}
	\left(
		e^{- c \kappa _\delta (n+1)^{-\alpha} n}
	\right)^{e^{-\lambda p^2 n}- \gamma _n}
	\\
	&= &
	\exp \left\{ 
		e^{(-\lambda - \alpha \delta - \varepsilon) n} 
		-
		c n \kappa _\delta (n+1)^{-\alpha}
		\left(e^{-\lambda p^2 n}- \gamma _n \right)
	\right\} . 
\end{eqnarray*}
The above exponential part is equal to 
\begin{eqnarray}
	&& 
	e^{(-\lambda -\alpha \delta - \varepsilon) n} 
	- 
	c n \left\{ e^{\delta (n+1)} a(n+1) \right\} ^{-\alpha} 
	\left( e^{-\lambda p^2 n} - e^{\eta _2n} a(n) \right) 
	\nonumber
	\\
	&=& 
	-c' n 
	\dfrac{e^{(-\alpha \delta - \lambda p^2 )n}}
		{a(n+1)^\alpha} 
	\left\{ 
		1 - e^{(\lambda p^2 + \eta _2)n} a(n)
	\right\} 
	\left[
		1 - 
		\dfrac{a(n+1)^\alpha 
			e^{(\lambda p^2 - \lambda - \varepsilon )n}}
		{c' n \{ 1 - e^{(\lambda p^2 + \eta _2)n} \} a(n)}
	\right] .
	\label{condition2}
\end{eqnarray}

According to \eqref{condition1} and \eqref{condition2}, 
we take $p \in (0,1)$ and $\eta_i$ ($0 <\eta_2 < \eta _1 < \delta$) 
which fulfill \eqref{eq:aimII}, as follows: 
\[
	- \alpha \eta _1 - \eta _2 - \lambda p <0; \quad 
	-\alpha \delta - \lambda p^2 >0 ; \quad 
	\lambda p^2 + \eta _2 < 0; \quad 
	\lambda p^2 - \lambda - \varepsilon < 0 ,
\]
that is, 
\begin{eqnarray}
	\left( -\lambda -\varepsilon \right) 
	\vee 
	\alpha \delta
	\vee 
	\eta _2
	< 
	-\lambda p^2; 
	\qquad 
	- \lambda p 
	<
	\alpha \eta _1 + \eta _2 .
	\label{eq:20230602-2}
\end{eqnarray}
Here, $\alpha \delta <-\lambda$, 
by the assumption $\delta < \frac{-\lambda}{\alpha}$. 
Since we finally take $\varepsilon \downarrow 0$, 
it is possible to assume that $\varepsilon$ is so small that 
$\varepsilon < \alpha \delta$ and 
$\alpha \delta < - \lambda - \varepsilon$. 
We firstly fix such $\varepsilon$. 
Then, we take $\eta_2>0$ such that 
\[
	- \lambda \sqrt{1 - \frac{\varepsilon}{-\lambda}} 
	- \alpha \delta 
	< \eta _2 < 
	-\lambda - \alpha \delta .
\]
For such given $\varepsilon$ and $\eta_2$, 
the condition \eqref{eq:20230602-2} reduces to 
\[
	1 - \frac{\varepsilon}{-\lambda} 
	< p^2; 
	\qquad 
	p< 
	\dfrac{\alpha \eta _1 + \eta _2}{-\lambda} 
	< \dfrac{\alpha \delta + \eta _2}{-\lambda} < 1 . 
\]
That is, 
\begin{equation}
	\sqrt{1 - \frac{\varepsilon}{-\lambda} }
	< p <
	\dfrac{\alpha \delta + \eta _2}{-\lambda} . 
	\label{condition3}
\end{equation}
Since $-\lambda \sqrt{1 - \varepsilon / (-\lambda)} - \alpha \delta < \eta _2$, 
we can choose $p \in (0,1)$ such as \eqref{condition3}. 
Under \eqref{condition3}, 
we see that \eqref{condition2} goes to $-\infty$ as $n \to \infty$,  
which implies the first part of \eqref{eq:aimII}. 
Additionally, 
\eqref{condition1} implies the second part of \eqref{eq:aimII}. 
Therefore, we have \eqref{eq:20230306-4}. 

By \eqref{eq:20230306-4} and the Borel-Cantelli lemma, 
\[
	\prob _x \left(
		\limsup _{n \to \infty} 
		\left\{
		\sum _{u \in Z_{n p}} \indi _{E_n ^u} 
		\le 
		e^{(-\lambda - \alpha \delta - \varepsilon) n} , 
		\ 
		N_{n p} \ge e^{-\lambda p^2 n}
		\right\}		
	\right) 
	= 0 . 
\] 
Here, we write $G_n$ for the component, 
then $\prob _x (\liminf _{n \to \infty} G_n^c) = 1$. 
We decompose $G_n^c$ as follows: 
\[
	G_n ^c 
	= 
	\left\{
		\sum _{u \in Z_{n p}} \indi _{E_n ^u} 
		> 
		e^{(-\lambda - \alpha \delta - \varepsilon) n} 
	\right\} 
	\cup 
	\left\{ 
		N_{n p} < e^{-\lambda p^2 n}
	\right\}		
	=: 
	G_n ^1 \cup G_n^2.
\]
By the subadditiviy, 
\[
	\prob_x \left( \left. 
		\liminf _{n \to \infty} G_n^c 
		\; \right| 
		M_\infty > 0
	\right)
	\le 
	\prob_x \left( \left. 
		\liminf _{n \to \infty} G_n^1 
		\; \right| 
		M_\infty > 0
	\right)	
	+
	\prob_x \left( \left. 
		\limsup _{n \to \infty} G_n^2 
		\; \right| 
		M_\infty > 0
	\right)
\]
Since 
$\prob _x (\liminf _{n \to \infty} G_n^c) = 1$, 
\begin{equation}
	1
	\le 
	\prob_x \left( \left. 
		\liminf _{n \to \infty} G_n^1 
		\; \right| 
		M_\infty > 0
	\right)	
	+
	\prob_x \left( \left. 
		\limsup _{n \to \infty} G_n^2 
		\; \right| 
		M_\infty > 0
	\right) .
	\label{eq:20230602-1}
\end{equation}

On the other hand, 
\begin{equation} 
	\left\{ \lim _{t \to \infty} \frac1t \log N_t = -\lambda \right\}
	\subset 
	\liminf_{n \to \infty} 
	\left\{ N_{np} \ge e^{-\lambda p^2 n} \right\} 
	=
	\left(
		\limsup _{n \to \infty} G_n ^2
	\right)^c .
	\label{eq:20230219-6}
\end{equation} 
We see from Lemma \ref{lem:3-7} and \eqref{eq:20230219-6} that
\[
	\prob _x \left( \left. 
		\left( 
			\liminf _{n \to \infty} G_n ^2
		\right) ^c
		\ \right| \ 
		M_\infty > 0
	\right)
	= 1 ,
\]
it follows that   
$$ 
	\prob _x \left( \left. 
		\limsup _{n \to \infty} G_n^2
		\ \right| 
		\ M_\infty > 0
	\right)
	= 
	\prob _x \left( \left. 
		\left( 
			\liminf _{n \to \infty} 
			\left\{ N_{np} \ge e^{-\lambda p^2 n} \right\}
		\right) ^c
		\ \right| \ 
		M_\infty > 0
	\right)
	= 0. 
$$
By the above and \eqref{eq:20230602-1}, 
we have 
$
	\prob_x \left( \left. 
		\liminf _{n \to \infty} G_n^1 
		\; \right| 
		M_\infty > 0
	\right)	= 1
$, namely, 
\begin{equation} 
	\prob _x \left( \left. 
		\liminf _{n \to \infty} 
		\left\{
		\sum _{u \in Z_{n p}} \indi _{E_n ^u} 
		> 
		e^{(-\lambda - \alpha \delta - \varepsilon) n} 
		\right\}
		\; \right| \; 
		M_\infty > 0
	\right)
	= 1 .
	\label{eq:20230306-6}
\end{equation} 
Since $\kappa _\delta (s)$ is monotone increasing, 
for any $\omega \in E_n^u$ and $s \in [n,n+1]$, 
\begin{eqnarray*} 
	\left| \bfX _s ^{(n p , u)} (\omega ) \right| 
	&\ge & 
	\left| \bfX _n ^{(n p , u)} (\omega ) \right| 
	- 
	\left| 
		\bfX _s ^{(n p , u)} (\omega ) - \bfX _n ^{(n p , u)} (\omega)
	\right|  
	\\
	&\ge & 
	2 \kappa_\delta (n+1) - \kappa _\delta (n) 
	\ge  
	\kappa _\delta (n+1) \ge \kappa _\delta (s).
\end{eqnarray*}
Thus, 
\begin{equation}
	E_n ^u 
	\subset 
	\left\{ 
		\left| \bfX _s ^{(n p , u)} \right| 
		> 
		\kappa _\delta (s) \text{ for all } s \in [n,n+1]
	\right\} .
	\label{eq:3-29}
\end{equation}
By using the inclusion relation \eqref{eq:3-29}, 
\begin{eqnarray*}
	N_t ^{\kappa _\delta (t)} 
	&=&
	\sum _{u \in Z_t} 
		\indi _{[\kappa _\delta (t) , \infty)} 
		\left( \left| \bfX _t ^u \right| \right)
	\\
	&\ge &  
	\sum _{u \in Z_{[t]p}} 
	\indi _{\left\{ \left| \bfX _s ^{([t]p,u)} \right| 
		> \kappa _\delta (s) \text{ for all } s \in 
	[[t], [t]+1] \right\} } 
	\\ 
	&\underset{\eqref{eq:3-29}}{\ge} & 
	\sum _{u \in Z_{[t]p}} \indi _{E_{[t]}^u} .
\end{eqnarray*}
Therefore, by \eqref{eq:20230306-6}, 
\[
	\liminf _{t \to \infty} 
	\frac1t \log 
	N_t ^{\kappa _\delta (t)} 
	\ge 
	\liminf _{t \to \infty}
	\frac{1}{[t]} \log \left( 
	\sum _{u \in Z_{[t]p}} 
	\indi _{E_{[t]} ^u}
	\right)
	> 
	- \lambda - \alpha \delta - \varepsilon, 
\] 
$\prob _x (\cdot \mid M_\infty > 0)$-a.s. 
Finally, letting $\varepsilon \downarrow 0$, 
we have our claim. 
\end{proof} 

We now prove Theorem \ref{theo:main}. 
\begin{proof}
(i) By Lemma \ref{lem:limsup}, 
for any $\delta > -\lambda /\alpha$, 
$$ 
	\begin{aligned}
	\limsup _{t \to \infty} 
	\left( N_t ^{\kappa _\delta (t)} \right) ^{1/t}
	&= 
	\limsup _{t \to \infty} 
	\exp \left( \frac{1}{t} \log N_t ^{\kappa _\delta (t)} \right)
	\\
	&=
	\exp \left( 
		\limsup _{t \to \infty}
		\frac{1}{t} \log N_t ^{\kappa _\delta (t)} 
	\right)
	\\ 
	&\le
	e^{-\lambda - \alpha \delta} < 1, \quad 
	\prob _x \text{-a.s.}
	\end{aligned}
$$ 
Thus, we have 
$\lim_{t \to \infty} N_t ^{\kappa_\delta (t)} = 0$, 
$\prob _x$-a.s. 

\noindent
(ii) Let $\delta \in (0, -\lambda / \alpha)$.  
By Lemma \ref{lem:limsup}, 
\[
	\limsup _{t \to \infty} 
	\frac1t \log 
	N _t ^{\kappa _\delta (t)} 
	\le 
	- \lambda - \alpha \delta , \qquad 
	\prob _x\text{-a.s.}
\]
By Lemma \ref{lem:liminf}, 
\[
	\liminf _{t \to \infty} 
	\frac1t \log 
	N _t ^{\kappa _\delta (t)} 
	\ge 
	- \lambda - \alpha \delta , \qquad 
	\prob _x (\cdot \mid M_\infty >0)\text{-a.s.}
\]
These show the second claim. 
\end{proof}

\begin{corollary}\label{cor}
We make the same assumptions as Theorem \ref{theo:main} 
and we assume $\prob_x (M_\infty >0) > 0$. 
Then,  
\[
	\lim _{t \to \infty} 
	\dfrac{1}{t} \log L_t
	= 
	\dfrac{-\lambda}{\alpha}, \qquad 
	\prob _x (\cdot \mid M_\infty > 0)\text{-a.s.}
\]
\end{corollary}

\begin{proof}
Let $\delta > - \lambda /\alpha$. 
By Theorem \ref{theo:main} (i), 
for any $\varepsilon \in (0,1)$, 
\[
	\prob _x \left( 
		\lim _{t \to \infty} 
		N_t ^{\kappa _\delta (t)} 
		\le
		 \varepsilon 
	\right)
	= 1  .
\]
That is, for almost all $\omega$, 
there exists $T_0 (\omega) >0$ such that  
$N_t ^{\kappa _\delta (t)} (\omega) \le \varepsilon$, 
for any $t > T_0 (\omega)$. 
That is, $N_t ^{\kappa _\delta (t)} (\omega) = 0$, 
for any $t > T_0 (\omega)$. 
It follows that 
all particles are contained in $B(\kappa _\delta (t))$, for $t > T_0 (\omega)$.  
Namely, 
\[
	L_t (\omega) \le \kappa _\delta (t) = e^{\delta t} a(t) , 
	\quad \text{for all } t  > T_0 (\omega) 
\]
and we have 
$\limsup _{t \to \infty} t^{-1} \log L_t \le \delta$, 
$\prob _x$-a.s. 
By letting $\delta \downarrow - \lambda / \alpha$, 
\begin{equation} 
	\limsup _{t \to \infty} 
	\frac1t \log L_t 
	\le 
	\dfrac{-\lambda}{\alpha}, 
	\quad 
	\prob_x \text{-a.s.}
	\label{eq:20230307-1}
\end{equation}
When $\delta \in ( 0 , - \lambda / \alpha)$, 
Lemma \ref{lem:limsup} provides that 
$N_t ^{\kappa _\delta (t)} \to \infty$, 
$\prob _x (\cdot \mid M_\infty >0)$-a.s. 
That is, the maximal displacement is greater than $\kappa _\delta (t)$: 
\[
	L_t \ge \kappa _\delta (t) = e^{\delta t} a(t), \quad 
	t \gg 1 , \quad \prob _x (\cdot \mid M_\infty >0)\text{-a.s.}
\] 
Therefore, 
\[
	\liminf _{t \to \infty} 
	\frac1t \log L_t \ge \delta , 
	\quad \prob _x (\cdot \mid M_\infty > 0)\text{-a.s.} 
\]
Taking $\delta \uparrow -\lambda / \alpha$, we obtain 
\begin{equation}
	\liminf _{t \to \infty} 
	\frac1t \log L_t
	 \ge 
	\dfrac{-\lambda}{\alpha}, 
	\quad 
	\prob_x (\cdot \mid M_\infty > 0)\text{-a.s.}
	\label{eq:20230307-2}
\end{equation}
According to \eqref{eq:20230307-1} and \eqref{eq:20230307-2}, 
we complete the proof. 
\end{proof}

\appendix

\section{Proof of Lemma \ref{lem:runningmax}}\label{appendix1}
More general cases appeared in \cite{KR}. 
Here, we use the fundamental argument (e.g., \cite[Section 2.8.A]{KS}).
\begin{proof}
Let $\kappa >0$. 
We see from $\{ |X_t| \ge \kappa \} \subset \{ \calM_t \ge \kappa \}$ 
that $P ( |X _t| \ge \kappa ) \le P ( \calM _t \ge \kappa )$,
and thus, 
the left side inequality of (\ref{eq:lem}) holds. 
We next show that the right side of (\ref{eq:lem}) holds. 
For any $|x| < \kappa$, 
\begin{eqnarray}
	P_x ( \calM _t \ge \kappa ) 
	&=&
	P_x ( \calM _t \ge \kappa , \ |X_t| \le \kappa )
	+
	P_x ( \calM _t \ge \kappa , \ |X_t| > \kappa )
	\nonumber
	\\
	&= &
	P_x ( \calM _t \ge \kappa , \ |X_t| \le \kappa )
	+
	P_x (|X_t| > \kappa) .
	\label{eq:divide}
\end{eqnarray} 
Let us set 
$T_\ell = \inf \left\{ t> 0 : |X_t| \ge \ell \right\}$, 
$\ell \ge 0$. 
By the strong Markov property,  
\begin{equation}
	P_x ( \calM _t \ge \kappa , \ |X_t| \le \kappa ) 
	= 
	P_x ( T_\kappa \le t , \ |X_t| \le \kappa ) 
	=
	E_x \left[
		\left. 
		P _{X_s} 
			\left( \left| X_{t - s} \right| \le \kappa \right) 
		\right|_{s = T_\kappa}
		\ ; \ T_\kappa \le t
	\right] . 
	\label{eq:20221027-1}
\end{equation} 
We see from the spatial homogeneity of the stable process that, 
for $|w| \ge \kappa$ and $u > 0$, 
\begin{equation}
	P_w ( |X_u| \le \kappa )
	\le 
	P_w ( |X_u| \ge \kappa ) .
	\label{eq:20221027-2}
\end{equation}
Indeed, if $|w| \ge \kappa$, 
then $\{ z : |2 w - z | \le \kappa \} \subset \{z : \kappa \le |z|\}$. 
Thus, 
\begin{eqnarray*}
	&& 
	P_w ( |X_u| \le \kappa )
	=
	\int _{|y| \le \kappa} 
		u^{-d/\alpha} g \left( \frac{|w-y|}{u^{1/\alpha}} \right)
	\, d y
	=
	\int _{|y| \le \kappa } 
		u^{-d/\alpha} g \left( \frac{|w-(2w -y)|}{u^{1/\alpha}} \right) 
	\, d y
	\\
	&=& 
	\int _{|2 w - z| \le \kappa } 
		u^{-d/\alpha} g \left( \frac{|w-z|}{u^{1/\alpha}} \right) 
	\, d z 
	\le 
	\int _{|z| \ge \kappa } 
		u^{-d/\alpha} g \left( \frac{|w-z|}{u^{1/\alpha}} \right)
	\, d z
	=
	P_w ( | X_u | \ge \kappa ) .
\end{eqnarray*}
Substituting $X_{T_\kappa}$ and $t-T_\kappa$ for $w$ and $u$ 
in \eqref{eq:20221027-2}, respectively, 
furthermore, combining (\ref{eq:20221027-1}) with (\ref{eq:20221027-2}), 
we obtain 
\[
	P_x ( \calM _t \ge \kappa , \ |X_t| \le \kappa ) 
	\le 
	E_x \left[
		\left. P _{X_s} 
			\left( \left| X_{t - s} \right| \ge \kappa \right) 
		\right| _{s = T_\kappa}
		\ ; \ 
		T_\kappa \le t
	\right] 
	= 
	P_x ( |X_t| \ge \kappa , \ T_\kappa \le t) .
\]
Since  $\{ |X_t| \ge \kappa \} \subset \{ T_\kappa \le t\}$, 
the right side above is equal to 
$ P (|X_t| \ge \kappa )$, that is,  
$$
	P_x ( \calM _t \ge \kappa , \ |X_t| \le \kappa ) 
	\le 
	P_x ( |X_t| \ge \kappa ) .
$$ 
We have that 
the first term on  the right side of (\ref{eq:divide}) is bounded above by 
$P_x (|X_t| \ge \kappa )$, and thus 
$
	P_x (\calM_t \ge \kappa ) 
	\le 
	2 P_x (|X_t| \ge \kappa )
$. 
This is the right-side inequality of (\ref{eq:lem}). 
\end{proof}

\renewcommand{\abstractname}{Acknowledgements}
\begin{abstract}
	The author thanks Professor Yuichi SHIOZAWA 
	for his helpful comments and suggestions.  
	The author also would like to thank Christopher B. Prowant 
	for carefully proofreading the manuscript. 
\end{abstract}


\begin{thebibliography}{99}


\bibitem{ABM91}
Albeverio, S., Blanchard, P.  and Ma, Z.: 
Feynman-Kac semigroups in terms of signed smooth measures,
in ``Random Partial Differential Equations'' (U. Hornung et al.\, Eds.),
Birkh\"auser, Basel, (1991), 1--31.

\bibitem{Bhatt}
Bhattacharya, A., Hazra, R. S. and Roy, P., 
Branching random walks, stable point processes and regular variation,
Stoch. Proc. Appl., 
{\bf 128} (1), (2018), 182--210. 

\bibitem{BH14} 
Bocharov, S. and Harris, S. C.: 
Branching Brownian motion with catalytic branching at the origin, 
Acta Appl. Math. {\bf 34},  (2014), 201--228. 

\bibitem{BH16} 
Bocharov, S. and Harris, S. C.: 
Limiting distribution of the rightmost particle in catalytic branching Brownian motion, 
Electron. Commun. Probab. {\bf 21}, no. 70, (2016), 12 pp. 

\bibitem{BW}
Bocharov, S. and Wang, L.: 
Branching Brownian motion with spatially homogeneous and point-catalytic branching, 
J. Appl. Probab. {\bf 56}, (2019), 891--917.

\bibitem{B78}
Bramson, M.~D.: 
Maximal displacement of branching Brownian motion,
Comm. Pure Appl. Math. {\bf 31}, (1978), 531--581.

\bibitem{Bulin2018}
Bulinskaya, E. V.: 
Maximum of a catalytic branching random walk, 
Russ. Math. Surv. {\bf 74} (3), (2018), 546--548. 


\bibitem{Bulin2021}
Bulinskaya, E. V.: 
Maximum of catalytic branching random walk with
regularly varying tails, 
J. Theoret. Probab. {\bf 34}, (2021), 141--161.


\bibitem{Durrett}
Durrett, R.:  
{\it Probability: Theory and Examples}, 
Cambridge Series in Statistical and Probabilistic Mathematics, Series Number 49, 
Cambridge University Press; 5th edition (2019). 


\bibitem{E84}
Erickson, K.~B.: 
Rate of expansion of an inhomogeneous branching process of Brownian particles,
Z. Wahrsch. Verw. Gebiete {\bf 66}, (1984), 129--140.  


\bibitem{FOT11} 
Fukushima, M., Oshima, Y. and Takeda,T.: 
{\it Dirichlet Forms and Symmetric Markov Processes}, 
2nd rev.\ and ext.\ ed., 
Walter de Gruyter (2011).

\bibitem{HaH}
Hardy, R. and Harris, S. C.:
A spine approach to branching diffusions with applications to
$L^p$-convergence of martingales, 
{\it S\'eminaire de Probabilit\'es}, XLII, (2009).



\bibitem{INW-I}
Ikeda, N., Nagasawa, M. and Watanabe, S.: 
Branching Markov Processes I, 
J. Math. Kyoto Univ. {\bf 8} (2), (1968), 233--278. 

\bibitem{INW-II}
Ikeda, N., Nagasawa, M. and Watanabe, S.: 
Branching Markov Processes II, 
J. Math. Kyoto Univ. {\bf 8} (3), (1968), 365--410. 

\bibitem{INW-III}
Ikeda, N., Nagasawa, M. and Watanabe, S.: 
Branching Markov Processes III, 
J. Math. Kyoto Univ. {\bf 9} (1), (1969), 95--160. 



\bibitem{KS}
Karatzas, I. and Shreve, S. E.: 
{\it Brownian Motion and Stochastic Calculus}, 2nd. ed., 
Springer New York, NY (2014)

\bibitem{KR}
K\"uhn, F. and Schilling, R. L.:  
Maximal inequalities and some applications, 
Probab. Surv. {\bf 20}, (2023), 382--485. 


\bibitem{LS88} 
Lalley, S. P. and Sellke, T.: 
Traveling waves in inhomogeneous branching Brownian motions. I, 
Ann. Probab. {\bf 16} (3), (1988), 1051--1062. 

\bibitem{LS89} 
Lalley, S. P. and Sellke, T.:  
Traveling waves in inhomogeneous branching Brownian motions II, 
Ann. Probab. {\bf 16}, (1988), 1051--1062. 

\bibitem{LSh}
Lalley, S. P. and Shao, Y.:  
Maximal displacement of critical branching symmetric stable processes, 
Ann. Inst. Henri Poincar\'e Probab. Stat. {\bf 52} (3), (2016), 1161--1177. 

\bibitem{Mc75}
McKean, H.~P.: 
Application of Brownian motion to the equation of Kolmogorov-Petrovskii-Piskunov, 
Comm. Pure. Appl. Math. {\bf 28}, (1975), 323--331. 

\bibitem{ShioandN}
Nishimori, Y. and Shiozawa, Y.: 
Limiting distributions for the maximal displacement of branching Brownian motions, 
J. Math. Soc. Japan {\bf 74} (1), (2022), 177--216. 

\bibitem{RSZ}
Ren, Y.-X., Son, R. and Zhang, R.: 
Weak convergence of the extremes of branching L?vy processes with regularly varying tails, 
arXiv:2210.06130

\bibitem{RY}
Revuz, D and Yor, M.:
{\it Continuous Martingales and Brownian Motion}, 
corrected third printing of the third edition, 
Springer Berlin, Heidelberg (2005)


\bibitem{S08}
Shiozawa, Y.: 
Exponential growth of the numbers of particles 
for branching symmetric $\alpha$-stable processes,
J. Math. Soc. Japan {\bf 60}, (2008), 75--116.

\bibitem{S18}
Shiozawa, Y.:
Spread rate of branching Brownian motions, 
Acta Appl. Math. {\bf 155}, (2018), 113--150. 

\bibitem{S22}
Shiozawa, Y.: 
Maximal displacement of branching symmetric stable processes, 
Dirichlet Forms and Related Topics, 
Springer Proceedings in Mathematics \& Statistics {\bf 394}, (2022), 461--491. 

\bibitem{T08}
Takeda, M.: 
Large deviations for additive functionals of symmetric stable processes, 
J. Theoret. Probab. {\bf 21}, (2008), 336--355. 

\bibitem{Wada}
Wada, M.: 
Asymptotic expansion of resolvent kernels 
and behavior of spectral functions for symmetric stable processes, 
J. Math. Soc. Japan {\bf69} (2), (2017) , 673-692. 

\end{thebibliography}
\end{document}